\documentclass[a4paper,11pt,reqno]{amsart}

\usepackage{fontenc}
\usepackage[english]{babel}
\usepackage{csquotes}

\usepackage{amsmath, amsfonts, amssymb, amsthm}
\usepackage{array}
\usepackage{bm}
\usepackage{braket}
\usepackage{derivative}
\usepackage{mathtools}
\usepackage{nicematrix}
\usepackage{paralist}
\usepackage{scalefnt}
\usepackage{tensor}
\usepackage{thmtools}
\usepackage{youngtab}

\usepackage[dvipsnames]{xcolor}
\usepackage{graphicx}
\usepackage[textwidth=0.8in]{todonotes}
\usepackage[normalem]{ulem}
\usepackage{extpfeil}
\definecolor{darkred}{RGB}{192,0,0}

\usepackage[colorlinks]{hyperref}
\hypersetup{
	colorlinks=true,
	linkcolor=orange!80!black,
	filecolor=magenta,      
	urlcolor=RoyalBlue,
	citecolor=RoyalBlue,
}
\usepackage[nameinlink]{cleveref}

\usepackage{tabularx}
\usepackage{adjustbox}
\usepackage{nicematrix}
\usepackage{listings}
\usepackage{varwidth}

\usepackage{tikz}
\usepackage{tikz-cd}
\usepackage{tikz-3dplot}
\usetikzlibrary{cd, arrows, matrix, positioning}

\usepackage[left=1in, right=1in, top=1in, bottom=1in, includefoot, headheight=13.6pt]{geometry}
\usepackage[font=small,margin=1.5cm]{caption}
\allowdisplaybreaks

\usepackage{csquotes}
\usepackage[
backend=biber,
style=alphabetic,
maxnames=99,
giveninits=true,
uniquename=init,
doi=false,isbn=false,url=false,eprint=false
]{biblatex}
\DeclareNameAlias{sortname}{given-family}
\DeclareNameAlias{default}{given-family}

\DeclareFieldFormat[article]{volume}{\textbf{#1}} %
\DeclareFieldFormat[article]{number}{no.~#1} %
\renewbibmacro*{volume+number+eid}{%
	\printfield{volume}%
	\setunit{\addspace}%
	\printtext[parens]{\printfield{year}}%
	\setunit{\addcomma\space}%
	\printfield{number}%
	\setunit{\addcomma\space}%
	\printfield{eid}%
}
\renewbibmacro*{issue+date}{} %

\AtEveryBibitem{%
	\ifentrytype{article}{\renewbibmacro{in:}{}}{}%
}

\usepackage{enumitem} %

\numberwithin{equation}{section}
\newtheorem{theorem}[equation]{Theorem}
\numberwithin{theorem}{section}
\newtheorem{lemma}[theorem]{Lemma}
\newtheorem{corollary}[theorem]{Corollary}
\newtheorem{proposition}[theorem]{Proposition}

\theoremstyle{definition}
\newtheorem{definition}[theorem]{Definition}

\newtheorem{remarkx}[theorem]{Remark}
\newenvironment{remark}
{\pushQED{\qed}\remarkx}
{\popQED\endremarkx}
\newtheorem{example}[theorem]{Example}
\newtheorem{algorithm}[theorem]{Algorithm}
\Crefname{assumption}{Assumption}{Assumptions}

\newcommand{\bbC}{\mathbb{C}}

\newcommand{\bbN}{\mathbb{N}}

\newcommand{\bbP}{\mathbb{P}}

\newcommand{\mfS}{\mathfrak{S}}

\newcommand{\mfm}{\mathfrak{m}}

\DeclareMathOperator{\brk}{brk}
\DeclareMathOperator{\Cat}{Cat}
\DeclareMathOperator{\CenOp}{Cen}

\DeclareMathOperator{\End}{End}

\DeclareMathOperator{\Hess}{Hess}
\DeclareMathOperator{\Hilb}{Hilb}
\DeclareMathOperator{\Hom}{Hom}

\DeclareMathOperator{\id}{id}
\DeclareMathOperator{\im}{im}

\DeclareMathOperator{\m}{m}

\DeclareMathOperator{\Proj}{Proj}

\DeclareMathOperator{\rk}{rk}

\DeclareMathOperator{\Soc}{soc}

\DeclareMathOperator{\Spec}{Spec}

\newcommand\restr[2]{{ %
		\left.\kern-\nulldelimiterspace #1 
		\vphantom{\big|} 
		\right|_{#2} 
}}
\newcommand{\transpose}[1]{\tensor[^{\mathrm{t}}]{#1}{}}
\renewcommand{\epsilon}{\varepsilon}
\newcommand{\cen}[1]{\CenOp_{#1}}
\newcommand{\bul}{%
	\mathrel{%
		\mathchoice
		{\raisebox{0.2ex}{\scalebox{0.6}{$\bullet$}}}
		{\raisebox{0.2ex}{\scalebox{0.6}{$\bullet$}}}
		{\raisebox{0.15ex}{\scalebox{0.5}{$\bullet$}}}
		{\raisebox{0.1ex}{\scalebox{0.4}{$\bullet$}}}%
	}%
}
\newcommand{\buls}{\bul\cdots\bul}

\title{Polynomials of minimal border rank}
\author{Cosimo Flavi}
\address[Cosimo Flavi]{Instytut Matematyczny Polskiej Akademii Nauk,
ul.~Śniadeckich 8, 00-656 Warsaw, Poland.}
\email{cflavi@impan.pl}
\author{Weronika Obcowska}
\address[Weronika Obcowska]{
	Wydział Matematyki, Informatyki i Mechaniki, Uniwersytet Warszawski,
	ul.~Stefana Banacha
	2, 02-097 Warsaw, Poland.
}
\email{wo430002@students.mimuw.edu.pl}
\author{Tim Seynnaeve}
\address[Tim Seynnaeve]{Instytut Matematyczny Polskiej Akademii Nauk,
ul.~Śniadeckich 8, 00-656 Warsaw, Poland.}
\email{tseynnaeve@impan.pl}



\begin{document}

\begin{abstract}
We use the correspondence between iterated multiplication tensors of Gorenstein algebras and homogeneous polynomials of minimal smoothable rank to classify polynomials of minimal border rank of sufficiently high degree in up to 7 variables.
\end{abstract}

\maketitle

\section{Introduction} \thispagestyle{empty}

The \emph{Waring rank} of a degree $d$ homogeneous polynomial $f$ %
is the smallest number $r$ such that $f$ can be written as a sum $f=\sum_{i=1}^{r}{\ell_i^d}$, where $\ell_1,\ldots,\ell_r$ are linear forms. The problem of finding such Waring decompositions traces back to the 19th century \cite{Hou77,Leb59,Luc76}, and has applications in signal processing \cite{Com94,Gre+02,SGB00,Sid+17} and complexity theory \cite{Lan17,Str83,Val01,Wig19}. Homogeneous polynomials are synonymous with symmetric tensors, and as such Waring rank is a symmetric analogue to tensor rank. In this paper, %
the terminology \emph{rank} will always refer to Waring rank (see also \Cref{rem:Comon}).

It may happen that a polynomial $f$ can be arbitrarily closely approximated by polynomials of Waring rank $f$, while $f$ itself has larger Waring rank. This gives rise to the notion of \emph{border Waring rank}. Geometrically, the set of polynomials of border rank at most $r$ is the $r$-th secant variety to a Veronese variety. For more details, see \cite{Ber+18,BGI11,CGO14,Landsberg12,LO13}. %

A polynomial in $n$ variables is \emph{concise} if it cannot be written in fewer than $n$ variables after a linear change of coordinates. It is an easy fact that a concise polynomial must have border rank at least $n$, and polynomials that achieve this bound are said to have \emph{minimal border rank}. Minimal border rank polynomials in up to $5$ variables were classified in 2010 by Landsberg and Teitler \cite{LT10}. However, as we discuss below, the classification they presents seems to contain some inaccuracies.

In this paper, we introduce a novel approach to classifying polynomials of minimal border rank, building on earlier work present in the literature. The key ideas are the following:
\begin{enumerate}
	\item The relation between border rank and \emph{smoothable rank}, as studied in~\cite{BB14}: when $d \geq r-1$, being of border rank $r$ is equivalent to lying in the linear span of a degree $r$ smoothable scheme embedded in the Veronese variety~\cite[Proposition 2.5]{BB14}. 
	\item The only concise tensors in $S^d\bbC^n$ that lie in the linear span of a degree $n$ subscheme of the Veronese are evaluation tensors of Gorenstein algebras. This is \Cref{thm:multTensorsAndLinearSpanOfSpec}, which is a symmetric version of the cactus apolarity lemma in~\cite{JJ26}, see also~\cite[Lemma 5.5]{DM26}. %
	\item The notion of the \emph{centroid} for bilinear maps (or equivalently: order $3$ tensors) was introduced in~\cite{Myasnikov} and further developed in~\cite{Wil12}. They appear in~\cite{JLP24} under the name of 111-algebras. Centroids can also be defined for tensors of arbitrary order, see~\cite{BMW20,CFJ25}. Given an evaluation tensor, we can recover the corresponding algebra by taking the centroid. We use this to show that two Gorenstein algebras have isomorphic iterated multiplication tensors if and only if the algebras are isomorphic (see \Cref{proposition:centroid_isomorphic_A}).
\end{enumerate}
Putting these ideas together, we find that as long as $d \geq n-1$, classifying minimal border rank polynomials in $S^d\bbC^n$ is equivalent to classifying $n$-dimensional smoothable Gorenstein algebras. For $n \leq 7$, there are only finitely many $n$-dimensional algebras up to isomorphism; they have been classified in \cite{Casnati}. This allows us in particular to recover (and correct) the lists of minimal border rank polynomials for $n \leq 5$ presented in \cite{LT10}, except for the case $S^3\bbC^5$, and to obtain a classification in the cases $n=6,7$ and $d\geq n-1$. The resulting list is presented in \Cref{tab:MainTable}. For the remaining cases $S^3\bbC^5$, $S^3\bbC^6$, $S^4\bbC^6$, $S^3\bbC^7$, $S^4\bbC^7$, $S^5\bbC^7$ we expect there to be additional minimal border rank polynomials, essentially because here smoothable and border ranks do not agree. One of these, already determined in \cite{JLP24}, is presented in \Cref{rmk:S3C5}. We leave the remaining cases for future work. 

As soon as $n \geq 8$, there are infinitely many nonisomorphic smoothable Gorenstein algebras of dimension $n$. This implies that, for any $n \geq 8$ and $d \geq 3$, there are infinitely many nonequivalent polynomials of minimal border rank of degree $d$ in $n$ variables. Note that this runs contrary to what one would expect based on a dimension count \cite[Remark 10.6]{LT10}. We provide an explicit such family in \Cref{ex:elliptic}. In fact, \cite{Casnati} also provides the classifications of $8$- and $9$-dimensional algebras, which involve infinite families depending on a parameter. Hence we can extend our classification of minimal border rank polynomials to $n \leq 9$. 

In \cite{JJ26} it is shown that a polynomial is the evaluation tensor of a Gorenstein algebra if and only if it is \emph{$1$-generic} and has a centroid of sufficiently high dimension. In the final subsection of the paper, we turn this characterization into an algorithm that can verify whether or not a given polynomial is the evaluation tensor of a Gorenstein algebra. Applying this procedure to \cite[Theorem 10.4]{LT10} reveals that two of the polynomials listed there in fact do not have minimal border rank, see \Cref{example:dim4table}.

\subsection*{Supplementary material} This paper is accompanied by supplementary material, consisting of code for the software Macaulay2~\cite{MS}. The code contains two algorithms, the first of which computes the evaluation tensor of a given Gorenstein algebra. Applying this algorithm to the list in~\cite{Casnati}, we obtain our classification of polynomials of minimal border rank. The second algorithm verifies whether a given polynomial is an evaluation tensor. For a more detailed description of the supplementary material, we refer to \autoref{Appendix}.

\subsection*{Acknowledgments}
The authors would like to thank Joachim Jelisiejew for suggesting to investigate this problem and for helpful discussions, and Jaros\l aw Buczy\'nski for helpful comments.

\subsection*{Notation and conventions} We work over the complex numbers $\bbC$. All vector spaces we consider are finite-dimensional until explicitly stated otherwise. The linear span of a subset $S$ of a vector space will be denoted $\langle S \rangle$. All algebras we consider are unital, associative, commutative $\bbC$-algebras.%

\section{Preliminaries}
\subsection{Symmetric tensors and their border ranks}
For any $n$-dimensional vector space $V$, a tensor $T\in V^{\otimes d}$ is \textit{decomposable} if $T=v_1\otimes\cdots\otimes v_d$ for some vectors $v_1,\dots,v_d\in V$. The \text{rank} of a tensor is the minimal number $r$ such that it can be written as a sum of $r$ decomposable tensors. In the case of a symmetric tensor $f\in S^dV$, the \textit{symmetric rank} of $f$ is the minimal number of tensor powers $v^{\otimes d}$ whose sum is equal to $f$.
Once a basis of an $n$-dimensional vector space $V$ is fixed, any symmetric tensor in $S^dV$ can be considered as a homogeneous polynomial of degree $d$ in $n$ variables. This is due to a canonical isomorphism between the $d$-th symmetric power of $V$ and the space $\bbC[V^*]_d$ of polynomial functions of degree $d$ on the dual space $V^*$ (see, e.g., \cite[Section 3.1]{CGLM08}). The Waring rank and symmetric rank are equivalent notions. %
For any $d\in\bbN$ we denote by $\nu_{d}$ the $d$-Veronese embedding
\[
	\nu_{d}\colon\bbP V \to \bbP(S^dV),\qquad
	[\ell] \mapsto {[{\ell}^d]},
\]
and $\nu_{d}(\bbP V)$ is the \textit{$d$-Veronese variety}. Clearly, we have $S^dV=\langle \nu_{d}(\bbP V)\rangle$.
Consider for every $r\in\bbN$ the set 
\[
\sigma_{r}^{\circ}\bigl(\nu_{d}(\bbP V)\bigr)=\Set{[g]\in \bbP(S^dV)|\rk g\leq r}.
\]
The \textit{border rank} of a polynomial $f\in S^dV$, denoted by $\brk f$, is the minimum  number $r\in\bbN$ such that $\smash{f\in\sigma_r(\nu_{d}\bigl(\bbP V)\bigr)\coloneqq\overline{\sigma_r^{\circ}(\nu_{d}(\bbP V))}}$, where by overline we mean the Zariski closure. 
We recall that the closure of $\sigma_{r}^{\circ}\bigl(\nu_{d}(\bbP V)\bigr)$ in the Zariski topology equals its closure  in the Euclidean topology (see \cite[Theorem 2.33]{Mum95}). Thus, we can think of the border rank of a polynomial $f\in S^d V$ as the minimal number $r\in\bbN$ such that
\[
f=\lim_{t\to 0}\sum_{j=1}^r\ell_j^d(t),
\]
where $\{\ell_j(t)\}_{t\in\bbC}$ is a family of linear forms for every $j=1,\dots,r$. 

We recall that the $i$-th flattening of a tensor $T\in V_1\otimes \cdots\otimes V_d$ is the associated linear map \[
T_i\colon V_i^*\to V_1\otimes\cdots\otimes V_{i-1}\otimes V_{i+1}\otimes\cdots\otimes V_d,
\]
for every $i=1,\dots,n$. A tensor $T\in V_1\otimes \cdots\otimes V_d$ is \textit{concise} if all the flattenings of $T$ are injective. 
If $T\in S^dV$, then $T_i=T_j$ for every $i,j=1,\dots,n$ and their image is contained in $S^{d-1}V\subseteq V^{\otimes d-1}$. In other words, a symmetric tensor has only one flattening $V^*\to S^{d-1}V$. The dual map $S^{d-1}V^*\to V$ is called the \emph{dual flattening}; clearly a symmetric tensor is concise if and only if its dual flattening is surjective.
The image of the dual flattening is the minimal subspace $U \subseteq V$ such that $T \in S^dU$, in particular a symmetric tensor is concise if there exists no proper subspace $U \subsetneq V$ such that $T \in S^d U$.

\begin{remark}\label{rmk:catalecticant} Flattenings of symmetric tensors have a natural interpretation in the laguage of polynomials. %
To fix notation, let $\{x_1,\dots,x_n\}$ be a basis of $V$, and denote the dual basis by $\{\alpha_1,\dots,\alpha_n\}$. %
In this notation, for a homogeneous polynomial $f\in\bbC[x_1,\dots,x_n]_d\cong S^dV$, its flattening is the gradient map
\[
	\nabla_{\!f}\colon\bbC[\alpha_1,\dots,\alpha_n]_1  \to  \bbC[x_1,\dots,x_n]_{d-1},\qquad
	\alpha_i  \mapsto {\displaystyle\pdv{f}{x_i}} =: \partial_if.
\]
In particular, a polynomial $f \in\bbC[x_1,\dots,x_n]_d$ is {concise} if and only if $\dim \langle \partial_1f,\ldots, \partial_nf \rangle = n$. %
\end{remark}
We can also consider more general matrixizations $S^kV^*\to S^{d-k}V$.  These correspond to the \textit{catalecticant maps}
\[\Cat_f^{\smash{(k)}}\colon \bbC[\alpha_1,\dots,\alpha_n]_k\to \bbC[x_1,\dots,x_n]_{d-k}\] defined on monomials by the relation
\[
\Cat_f^{\smash{(k)}}(\alpha_1^{j_1}\cdots \alpha_n^{j_n})=\partial_1^{j_1}\cdots \partial_n^{j_n} f.
\]
In particular, the dual flattening is given by the $(d-1)$-th %
catalecticant map.
In general, we have an inequality
\[
\rk f\geq\rk\bigl(\Cat_f^{\smash{(k)}}\bigr)
\]
for every $k\in\bbN$ (see, e.g., \cite[Proposition 2.43]{Fla25a} for a quick proof). This implies that, whenever a form $f\in S^dV$ is concise, we have $\rk f\geq n$ and, by upper semicontinuity of rank, also $\brk f\geq n$. When $\brk f=n$ we say that $f$ has \textit{minimal border rank}. 

\begin{remark} \label{rem:Comon}
	\emph{Comon's conjecture} states that the Waring rank of a symmetric tensor equals its tensor rank, and an analoguous statement for border Waring and border tensor rank has also been conjectured. While these conjectures are open in general, it has been proven in \cite{MV24} that in the regime we consider (concise tensors in $S^d\bbC^n$, with $d \geq n-1$), minimal border tensor rank is equivalent to minimal border Waring rank. 
\end{remark}

\begin{remark} \label{rmk:divided_power_basis}
	The natural pairing $V^* \otimes V \to \bbC$ induces a pairing $S^dV^* \otimes S^dV \to \bbC$. After choosing a basis of $V$ as in \Cref{rmk:catalecticant}, this pairing becomes
	\[
	\bbC[\alpha_1,\dots,\alpha_n]_d \otimes \bbC[x_1,\dots,x_n]_{d} \to \bbC,\qquad
	\alpha_1^{j_1}\cdots \alpha_n^{j_n} \otimes f  \mapsto \partial_1^{j_1}\cdots \partial_n^{j_n} f.
	\]
	In particular, note that the monomial bases of the two spaces are not dual to each other. The dual basis to the monomial basis of $S^dV^*$ is given by \emph{divided power monomials}:
	\[
	x_1^{[j_1]} \cdots x_n^{[j_n]} \coloneqq \frac{x_1^{j_1} \cdots x_n^{j_n}}{j_1! \cdots j_n!}.
	\]
	See also \cite[Remark 3.1]{BB14}.
	When writing polynomials in practice, it will often be useful to write them in the divided power basis. 
\end{remark}

The goal of this paper is to classify symmetric tensors of minimal border rank up to equivalence. Here by \textit{equivalence} (also called \textit{isomorphism} in the literature, see e.g.~\cite{JLP24}) we mean the following:
\begin{definition}
Given %
tensors $T\in V_1\otimes\cdots\otimes V_d$ and $S\in W_1\otimes \cdots\otimes W_d$, we say that $S$ is a \textit{restriction} of $T$, and we write $S\leq T$, if there exist $d$ linear maps $F_i\colon V_i\to W_i$, for $i=1,\dots,d$, such that $S=(F_1\otimes\cdots\otimes F_d)(T)$. If both $S\leq T$ and $T\leq S$, then we say that $S$ and $T$ are \textit{equivalent}
and we write $S\sim T$. 
\end{definition} 
\begin{remark}
	If $T$ and $S$ are concise tensors, then $S$ is equivalent to $T$ if and only if there exist \emph{isomorphisms} $F_i:V_i \to W_i$ such that $S=(F_1\otimes\cdots\otimes F_d)(T)$. If furthermore $T$ and $S$ are symmetric, these isomorphisms can be chosen to be equal, i.e.\ two concise symmetric tensors $T \in S^dV$ and $S \in S^dW$ are equivalent if and only if there exists an isomorphism $F:V \to W$ such that $S=F^{\otimes d}(T)$. In the language of polynomials, two concise homogeneous polynomials are equivalent if they agree up to a linear change of variables. {We leave these statements for the reader to check.}
\end{remark}

\subsection{Gorenstein algebras}
For our purposes, we need to recall some basic notions of algebras. This exposition is mainly based on \cite[Section 2.1]{Jel22}, see also \cite[Chapter 21]{Eis95}. By the term \textit{algebra} we mean an associative commutative $\bbC$-algebra with identity element. The \textit{degree} of an algebra $A$ is its dimension as a $\bbC$-vector space, denoted by $\dim_{\bbC} A$. Analogously, for any module $M$, its $\bbC$-dimension $\dim_{\bbC} M$ is called the \textit{degree} of $M$. We always assume $A$ to be a \emph{finite} $\bbC$-algebra, i.e.\ $\dim_{\bbC} A < \infty$.
Note that this implies that the Krull dimension of $A$ is zero.

If $M$ is an $A$-module and $\dim_{\bbC} M<\infty$, then it is possible to define the \textit{dual module} of $A$ as the $A$-module $M^*\coloneqq\Hom_{\bbC}(M,\bbC)$, where the multiplication operation is defined by
\[
(a\cdot\varphi)(m)\coloneqq \varphi(a\cdot m),
\]
for all $\varphi\in M^*$, $a\in A$, and $m\in M$. In the particular case where $M=A$, the module $A^*$ is called the \textit{canonical} $A$-module, also denoted by $\omega_A$.
\begin{definition}
	An algebra $A$ is \textit{Gorenstein} if $\omega_A$ is cyclic or, equivalently, if $\omega_A\cong A$ as $A$-modules. In this case any $\varepsilon \in \omega_A$ such that $A \cdot \varepsilon = \omega_A$ is called a \emph{dual generator} of $A$. %
\end{definition} 
A finite algebra $A$ can be uniquely (up to isomorphism) written as a product of finite local algebras; these local algebras are exactly the localizations of $A$ at its maximal ideals (see for instance \cite[Theorem 8.7]{AM69}). Moreover, $A$ is Gorenstein if and only if all its localizations at maximal ideals are Gorenstein.
The following characterization of local Gorenstein algebras will be useful for our computations in \Cref{sec:classification}. Recall that for a local ring $A$ with maximal ideal $\mfm$, the \emph{socle} is defined to be the ideal $\Soc A \coloneqq \Set{a \in A | a\cdot \mfm = 0}$.
\begin{proposition} \label{prop:local_Gorenstein_socle}
	For a finite local $\bbC$-algebra $A$, the following are equivalent:
	\begin{enumerate}
		\item\label{it:locGor1} $A$ is Gorenstein;
		\item\label{it:locGor2} $\Soc{A}$ is a one-dimensional $\bbC$-vector space;
		\item\label{it:locGor3} $\Soc{A}$ is a principal ideal.
	\end{enumerate}
	If the above conditions hold, then $\epsilon \in \omega_A$ is a dual generator if and only if $\epsilon(\Soc{A}) \neq 0$. 
\end{proposition}
\begin{proof}
	The equivalence \eqref{it:locGor1} $\Leftrightarrow$ \eqref{it:locGor2} is contained in \cite[Proposition 2.9]{Jel22} and the final statement is \cite[Proposition 2.11]{Jel22}. The equivalence \eqref{it:locGor2} $\Leftrightarrow$ \eqref{it:locGor3} follows by writing $A \cong \bbC \oplus \mfm$: if $a \in \Soc{A}$, then $A\cdot a = (\bbC \oplus \mfm)\cdot a = \bbC\cdot A$.
\end{proof}

	\subsection{Hilbert schemes of points}
	The assignment $R \mapsto \Spec R$ gives a one-to-one correspondence between finite algebras of degree $r$ and zero-dimensional $\bbC$-schemes of degree $r$.
	The \emph{Hilbert scheme} $\Hilb_r(\bbP V)$ parameterizes subschemes of degree $r$ of a given projective space $\bbP V$. 
	{In what follows, the notation $\langle Z \rangle$ denotes the \emph{scheme-theoretic linear span} of a subscheme $Z \subseteq\bbP V $. This is the linear subscheme of $\bbP V$ defined by the degree $1$ part of the homogeneous ideal $I(Z)$ of $Z$.}

	A subscheme of $\bbP V$ of degree $r$ is smooth if it is a union of $r$ distinct points (i.e.\ $R=\bbC \times \cdots \times \bbC$), and \emph{smoothable} if it is a flat limit of smooth subschemes. Smoothable subschemes form an irreducible component of $\Hilb_r (\bbP V)$, known as the \emph{smoothable component}. 
	By definition, the symmetric rank of $[f] \in \bbP S^dV$ is the minimal $r$ such that $[f] \in \langle\nu_d(Z)\rangle$, where $Z \subseteq \bbP V$ is a smooth scheme of degree $r$. If we replace $Z$ by a smoothable degree $r$ scheme (respectively any degree $r$ scheme) we obtain the definition of \emph{smoothable rank} (respectively \emph{cactus rank}). 
	
	By taking limits, one can see that if $[f] \in \langle\nu_d(Z)\rangle$ for some smoothable degree $r$ scheme, then $[f] \in \sigma_r\bigl(\nu_d(V)\bigr)$. In other words, border rank is upper bounded by smoothable rank. In general the inequality might be strict; however, the following proposition shows it is an equality if $r \leq d+1$.
	\begin{proposition}[{\cite[Proposition 2.5]{BB14}}] \label{prop:BB14}
		If $r \leq d+1$, then 
		\[
		\sigma_r(\nu_d(\bbP V)) = \bigcup\{\langle\nu_d(Z)\rangle \mid Z \subset \bbP V \text{ smoothable of degree at most } r\}.
		\]
	\end{proposition}
	Spectra of Gorenstein algebras form an open subscheme of $\Hilb_r(\bbP V)$, called the \emph{Gorenstein locus} $\Hilb^G_r (\bbP V)$. The main result of \cite{CJN15} states that for $r \leq 13$, $\Hilb^G_r (\bbP V)$ is irreducible. In other words, Gorenstein schemes of degree at most $13$ are automatically smoothable. {Aside from the aforementioned references, the reader who wishes to learn more about Hilbert schemes of points can consult \cite{Jel22,Jel24}}.
	
	\begin{remark}
		One can show that smoothability of a zero-dimensional subscheme $Z \subset \bbP V$ depends only on the isomorphism type of $Z$, not on the embedding (see, e.g., \cite[Theorem 3.16]{BJ17}). This means that we can without ambiguity define an algebra $A$ to be \emph{smoothable} if some/every embedding $\Spec A \hookrightarrow \bbP V$ is smoothable.
	\end{remark}
	
	\subsection{Centroids of tensors} 
	The action of $X_i\in\End V_i=V_i^*\otimes V_i$ on $V_1\otimes \cdots\otimes V_d$, for any $i=1,\dots,d$, is defined on decomposable tensors by the relation
	\[
	X_i\circ_i (a_1\otimes\dots\otimes a_d)=a_1\otimes\cdots\otimes a_{i-1}\otimes X_i(a_i)\otimes a_{i+1}\otimes\cdots\otimes a_d
	\]
	for all $a_j\in V_j$ and $j=1,\dots,d$. Considering $T$ and $X_i\circ_i T$ as multilinear maps
	\[
	V_1^*\times \cdots\times V_k^*\to V_{k+1}\otimes\cdots\otimes V_d,
	\]
	for any $k=1,\dots,d$,
	it is defined by
	\begin{equation}\label{definition:comp_i}
	(X_i\circ_i T)(\phi_1,\dots,\phi_k)=
	\begin{cases} T\bigl(\phi_1,\dots,\phi_{i-1},\transpose{X}_i(\phi_i),\phi_{i+1},\dots,\phi_k\bigr),\quad &\text{if $1\leq i\leq k$},\\
		X_i\circ_i \bigl(T(\phi_1,\dots,\phi_k)\bigr),\quad &\text{if $k+1\leq i\leq d$,}
	\end{cases}
	\end{equation}
	for all $\phi_j\in V_j^*$ with $j=1,\dots,k$. %
	\begin{definition} \label{def:centroid}
		For any $T\in V_1\otimes\cdots\otimes V_d$, the subspace
		\[
		\cen{T}=\Set{(X_1,\dots, X_d)\in \End(V_1)\times\cdots\times\End(V_d)|X_1\circ_1 T=\dots=X_d\circ_d T}
		\]
		is called the \textit{centroid} of $T$.
	\end{definition}
	For the case of partially symmetric tensors, the definition can be simplified by the use of the following lemma (see \cite[Lemma 2.11]{CFJ25} for a direct proof).
	\begin{lemma}\label{lemma:partiallySymmetricCentroids1}
		Let $e\in\bbN$, let $d_1,\dots,d_e\in\bbN$ such that $d_1 +  \cdots + d_e\geq 3$, let $T\in S^{d_1}V_1\otimes \cdots \otimes
		S^{d_e} V_e$ be a concise tensor, and let $(X_1, \ldots ,X_{d_1 +  \cdots + d_e})\in\cen{T}$. Then \[
		X_{d_1+\cdots+d_{j-1}+1} =X_{d_1+\cdots+d_{j-1}+2}=  \cdots = X_{d_1+\cdots+d_{j}},
		\]
		for every $j=1,\dots,d$.
	\end{lemma}
	For this reason, we will denote elements in $\cen{T}$, for $T\in S^{d_1}V_1\otimes \cdots \otimes
	S^{d_e} V_e$, as tuples $(Y_1,\ldots,Y_e)$ with $Y_j \in \End(V_j)$. The connection with the previous notation is given by $Y_j=X_{d_1+\cdots+d_j}$.
{Note that for $T\in S^{d_1}V_1\otimes \cdots \otimes
	S^{d_e} V_e$ and $Y_j \in \End(V_j)$, the tensors $Y_j \circ_{d_1+\cdots+d_{j-1}+k}T$ need not lie in $S^{d_1}V_1\otimes \cdots \otimes
	S^{d_e} V_e$.}
	In the partially symmetric case, it suffices to verify a subset of the equalities in \Cref{def:centroid}. This is stated more precisely in the following lemma, which is a slight variation of \cite[Lemma 2.11]{CFJ25}.
	\begin{lemma}\label{lemma:partiallySymmetricCentroids2}
		Let $T\in S^{d_1}V_1\otimes\cdots\otimes S^{d_e}V_e$. Then $(Y_1,\ldots,Y_e) \in \cen{T}$ if and only if both of the following conditions hold:
		\begin{enumerate}[label=(\arabic*)]
			\item\label{point 1} $Y_1 \circ_{1} T = Y_2 \circ_{d_1+1} T = \cdots = Y_e \circ_{d_1+\cdots+d_{e-1}+1} T$;
			\item\label{point 2}
			$Y_j \circ_{d_1+\cdots+d_{j-1}+1}T=Y_j \circ_{d_1+\cdots+d_{j-1}+2}T$ for every $j=1,\ldots,e$ such that $d_j > 1$.
		\end{enumerate}
	\end{lemma}
	\begin{proof}
	If $(Y_1,\ldots,Y_e) \in \cen{T}$, then both conditions are satisfied by definition. Conversely, let us assume that both conditions hold. By \ref{point 1}, it suffices to show that for every $j=1,\dots,e$ and every $k=1,\dots,d_j$, we have 
	\[
	Y_{j}\circ_{d_1+\cdots+d_{j-1}+1} T= Y_j\circ_{d_1+\cdots+d_{j-1}+k} T.
	\]
	By \ref{point 2}
 we just need to prove it for $k\neq 1,2$. Considering the transposition 
 \[
 \tau_k\coloneqq\bigl((d_1+\cdots+d_{j-1}+2)(d_1+\cdots+d_{j-1}+k)\bigr)\in\mfS_{d_1+\cdots+d_e},
 \]
 we have
 \begin{align*}
 Y_j\circ_{d_1+\cdots+d_{j-1}+1} T&= Y_j\circ_{d_1+\cdots+d_{j-1}+1} \tau_k(T)=\tau_k(Y_j\circ_{d_1+\cdots+d_{j-1}+1} T)\\
 &=\tau_k(Y_j\circ_{d_1+\cdots+d_{j-1}+2} T)=Y_j\circ_{d_1+\cdots+d_{j-1}+k} \tau_k(T)\\
 &=Y_j\circ_{d_1+\cdots+d_{j-1}+k} T.\qedhere
 \end{align*}
 \end{proof}
	 The following lemma gives us a practical way to compute the centroid in the fully symmetric case, where the tensor $f \in S^d V\cong\bbC[x_1,\dots,x_n]_d$ can be viewed as a polynomial. In what follows we denote the Hessian matrix of $f$ by $\Hess(f)$.
	\begin{lemma}\label{lem:centroidHessian}
 Let $f \in S^d V$ be a symmetric tensor. Then $Y \in \End V$ is in the centroid of $f$ if and only if $Y\cdot\Hess(f)$ is a symmetric matrix.%
	\end{lemma}
	\begin{proof}
		The claim is equivalent to $Y \cdot\Hess(f) = \Hess(f)\cdot \transpose{Y}$. Choose a basis $\{x_1, \ldots, x_n\}$ of $V$. Denote by $(a_{ij})_{i,j}$ the matrix of $Y$ in this basis. The Hessian of $f$ is given by $(\partial_i\partial_j f)_{i,j}$. The $(i,j)$-entry of $Y\cdot \Hess(f)$ is 
		\begin{equation*}
			(a_{i1}, \ldots, a_{in}) \cdot \transpose{(\partial_1\partial_j f, \ldots, \partial_n\partial_j f)} = a_{i1} \partial_1\partial_j f + \cdots + a_{in} \partial_n\partial_j f
		\end{equation*}
		and the $(i,j)$-entry of $\Hess(f) \cdot \transpose{Y}$ is 
		\begin{equation*}
			(\partial_i\partial_1 f, \ldots, \partial_i\partial_n f) \cdot \transpose{(a_{1j}, \ldots, a_{nj})} = a_{1j} \partial_i\partial_1 f + \cdots + a_{nj} \partial_i\partial_n f.
		\end{equation*}
		By \Cref{lemma:partiallySymmetricCentroids2}, $Y \in \cen{f}$ if and only if the following compositions are equal:
		\begin{equation}\label{r:equAidid}
			S^d V \xrightarrow{\phi} V \otimes V \otimes S^{d-2} V \xrightarrow{Y \otimes \id \otimes \id} V \otimes V \otimes S^{d-2} V,
		\end{equation}
		\begin{equation}\label{r:equidAid}
			S^d V \xrightarrow{\phi} V \otimes V \otimes S^{d-2} V \xrightarrow{\id \otimes Y \otimes \id} V \otimes V \otimes S^{d-2} V, 
		\end{equation}
		where the inclusion map $\phi \colon S^d V \xrightarrow{} V \otimes V \otimes S^{d-2} V$ is defined by 
		\begin{equation*}
			f \mapsto \frac{1}{d(d-1)} \sum_{i,j=1}^n x_i \otimes x_j \otimes \partial_i \partial_j f.
		\end{equation*}
		The image of $f$ under the map in (\ref{r:equAidid}) is 
		\begin{equation*}
			\frac{1}{d(d-1)} \sum_{i,j,\ell=1}^n a_{\ell i} x_\ell \otimes x_j \otimes \partial_i \partial_j f.
		\end{equation*}
		Note that $\sum_{i} a_{\ell i} \partial_i \partial_j f$ is the $(\ell,j)$-entry of $Y\cdot \Hess(f)$. On the other hand, the image of $f$ under the map in (\ref{r:equidAid}) is 
		\begin{equation*}
			\frac{1}{d(d-1)} \sum_{i,j,\ell=1}^n x_i \otimes a_{\ell j} x_\ell \otimes \partial_i \partial_j f, 
		\end{equation*}
		which for all $i,\ell=1,\dots,n$ gives the $(\ell,i)$-entry of $\Hess(f) \cdot \transpose{Y}$, that is, $\sum_{j} a_{\ell j} \partial_i \partial_j f$. The assertion follows by reindexing the sums. 
	\end{proof}

In general, the centroid $\cen{T}$ is a unital subalgebra of $\End(V_1)\times\cdots\times\End(V_d)$. If $T$ is concise and $d \geq 3$, it is commutative. 
The proof of this fact is given in \cite[Theorem 1.11]{JLP24} for the case where $d=3$ and $\dim V_1=\dim V_2=\dim V_3$, that is, tensors in $(\bbC^m)^{\otimes 3}$ for any $m\in\bbN$. This justifies, in particular, the choice of the name \textit{$111$-algebras}. The same proof can be easily extended to the general case (see also \cite{JelBedlewo}). %

\subsection{1-generic tensors}
In \cite{JLP24}, multiplication tensors %
of algebras are characterized in terms of \emph{1-genericity}. Following \cite[Proposition 4.13]{JJ26}, one can generalize this notion to higher order tensors.
\begin{definition}
	Let $T\in V_1\otimes\cdots\otimes V_d$ with $\dim V_i=m$ for every $i=1,\dots,d$. Then $T$ is \textit{$1_{V_i}$-generic} if there exists an element $\eta \in V_i^*$ such that $T_i(\eta)\in V_1\otimes\cdots\otimes V_{i-1}\otimes V_{i+1}\otimes \cdots\otimes V_d$ is concise. If $T$ is $1_{V_i}$-generic for every $i=1,\dots,d$, then $T$ is \textit{$1$-generic}.
\end{definition}
Note that a symmetric tensor is $1_{V_i}$-generic for some $i$ if and only if it is $1$-generic. In the language of polynomials, we can verify $1$-genericity using the following proposition:
\begin{proposition}\label{proposition_HMV20}
	Let $f \in S^dV$ be a concise symmetric tensor. Then the following are equivalent:
	\begin{enumerate}[label=(\roman*)]
		\item \label{it:HMV1} the Hessian matrix $\Hess(f)$ has nonzero determinant;
		\item \label{it:HMV2} there exists $\eta\in V^*$ such that the matrix $\Cat_f^{\smash{(d-2)}}(\eta^{d-2}) \in S^2V$ is full-rank. %
		\item \label{it:HMV3} $f$ is $1$-generic.
	\end{enumerate}
\end{proposition}
\begin{proof}
	The equivalence \ref{it:HMV1} $\Leftrightarrow$ \ref{it:HMV2} is {\cite[Lemma 4.7]{HMV20}}. We now show \ref{it:HMV2} $\Rightarrow$ \ref{it:HMV3}. %
	Suppose for contradiction that $f$ is not $1$-generic. Then for every $\eta \in V^*$, there exist a strict subspace $U \subsetneq V$ such that $\Cat_f^{\smash{(1)}}(\eta^{d-2}) \in S^{d-1}U$. Then also $\Cat_f^{\smash{(d-2)}}(\eta^{d-2}) \in S^2U$, which implies $\smash{\rk\left(\Cat_f^{\smash{(d-2)}}(\eta^{d-2})\right)}<\dim V$. Since $\eta \in V^*$ was chosen arbitrarily, we reach a contradiction with \ref{it:HMV2}. 
	To prove \ref{it:HMV3} $\Rightarrow$ \ref{it:HMV1}, write $n=\dim V$ and note that a polynomial is 1-generic if and only if there exist coefficients $a_1,\dots,a_n\in\bbC$ such that $\sum_j a_j\partial_jf$ is concise. By \Cref{rmk:catalecticant}, this implies that $\sum_{i,j} b_ia_j\partial_i\partial_jf\neq 0$ for all $b_1,\dots,b_n\in\bbC$. It follows that the map
	\[
	H\colon \bbC^n\to S^{d-2}V,\qquad (b_1,\dots,b_n)\mapsto (b_1,\dots,b_n)\cdot\Hess(f),
	\]
	is injective, so that $\det\bigl(\Hess(f)\bigr)\neq 0$.
\end{proof}

In \cite{JJ26} it is shown that 1-generic tensors correspond to iterated multiplication tensors of Gorenstein algebras. In \Cref{sec:classification}, we will use this to give a practical way of checking, under certain constraints, whether or not a given symmetric tensor has minimal border rank. 

\section{Evaluation tensors of Gorenstein algebras}

\subsection{Iterated multiplication tensors and evaluation tensors}
In this section we will consider elements of the symmetric product $S^dA$. To avoid ambiguity, we denote the product of two elements $a$ and $b$ of an algebra $A$ by $a\bul b \in A$, reserving the notation $ab$ for the symmetric product $ab \in S^2A$. 
	For any algebra $A$, 
	the \textit{$d$-th iterated multiplication tensor} is $\mu_{A}^{\smash{(d)}}\in(A^*)^{\otimes d}\otimes A$ defined by the multilinear map
	\[
		\mu_{A}^{(d)}\colon A\times\cdots\times A\to A,\qquad
		(a_1,\dots,a_d) \mapsto a_1\bul \cdots \bul a_d.
	\]
	Since $A$ is commutative, we have in fact $\mu_{A}^{\smash{(d)}} \in S^d{A^*} \otimes A$, i.e.\ the tensor is partially symmetric. 
	Given an element $\epsilon\in A^*$, %
	the composition $\epsilon\circ \mu_{A}^{\smash{(d)}}\colon A^{d} \to \bbC$ corresponds to a symmetric tensor, called \emph{evaluation tensor} and denoted by $\smash{\mu_{A,\epsilon}^{\smash{(d)}}}\in S^{d}A^*$. 

	If, as in \Cref{rmk:catalecticant}, we choose a basis $\{b_1,\dots,b_n\}$ of $A$ and identify $S^dA^*$ with $\bbC[x_1,\ldots,x_n]_d$, %
	then $\mu_A^{\smash{(d)}}\in S^dA^*\otimes A$ is the partially symmetric tensor
		\begin{align}
		\nonumber\mu_A^{\smash{(d)}}&=\sum_{i_1,\dots,i_d=1}^nx_{i_1}\otimes\cdots\otimes x_{i_d}\otimes (b_{i_1}\bul\cdots\bul b_{i_d})\\
		&=\sum_{\substack{0\leq j_1,\dots,j_n\leq d\\j_1+\cdots+j_n=d}}\binom{d}{j_1,\dots,j_n}x_1^{j_1}\cdots x_n^{j_n}\otimes (b_{1}^{\bul j_1}\bul\cdots\bul b_{n}^{\bul j_n}). \label{eq:mult_tensor_explicit}
	\end{align}
	and $\mu^{(d)}_{A,\epsilon}$ is the polynomial
\begin{equation}\label{eq:mult_tensor_as_polynomial}
\smash{\mu_{A,\epsilon}^{\smash{(d)}}}=\sum_{\substack{0\leq j_1,\dots,j_n\leq d\\j_1+\cdots+j_n=d}}\binom{d}{j_1,\dots,j_n}\epsilon(b_{1}^{\bul j_1}\bul\cdots\bul b_{n}^{\bul j_n})x_{1}^{j_1}\cdots x_{n}^{j_n}.
\end{equation}
After rescaling, this polynomial can be more succinctly written in the divided power basis: %
\begin{equation}\label{eq:mult_tensor_as_polynomial_div_power}
	d!\smash{\mu_{A,\epsilon}^{\smash{(d)}}}=\sum_{\substack{0\leq j_1,\dots,j_n\leq d\\j_1+\cdots+j_n=d}}\epsilon(b_{1}^{\bul j_1}\bul\cdots\bul b_{n}^{\bul j_n})x_{1}^{[j_1]}\cdots x_{n}^{[j_n]}. %
\end{equation}
\begin{remark}
The $k$-th catalecticant map of $\smash{\mu_{A,\epsilon}^{\smash{(d)}}}$ is the linear map
\begin{equation}\label{formula:flattenings_multiplication_maps}
T_{\mu_{A,\epsilon}^{(d)}}^{(k)}\colon S^kA\to S^{d-k}A^*,\qquad a_1 \cdots a_{k}\mapsto \mu_{A,(a_1\buls a_k)\cdot\epsilon}^{\smash{(d-k)}}
\end{equation}
In particular, from the point of view of polynomials, once we fix a basis, the partial derivative
\[
\partial_1^{j_1} \cdots \partial_n^{j_n} \mu_{A,\epsilon}^{\smash{(d)}}
\]
is the polynomial associated to the multiplication tensor
\[
{\mu_{A,(b_1^{\bul j_1}\buls b_n^{\bul j_n})\cdot\epsilon}^{\smash{(d-k)}}}
\]
for all $j_1,\dots,j_n\in\bbN$ such that $j_1+\cdots+j_n=k$.
\end{remark}

\begin{lemma}\label{lem:GorensteinTensor}%
	Let $A$ be a finite-dimensional $\bbC$-algebra and let $\epsilon \in A^*$. The following are equivalent:
	\begin{enumerate}[label=(\roman*)]
		\item \label{it1} $A$ is Gorenstein with dual generator $\varepsilon$;
		\item \label{it2} $\smash{\mu_{A,\epsilon}^{\smash{(d)}}}$ is concise for some $d \geq 2$;
		\item \label{it3} $\smash{\mu_{A,\epsilon}^{\smash{(d)}}}$ is concise for every $d\geq 2$.
	\end{enumerate}
	If these conditions hold, then $\smash{\mu_{A,\epsilon}^{\smash{(d)}}}$ is %
	{equivalent} to $\smash\mu_{A}^{\smash{(d-1)}}$. 
\end{lemma}
\begin{proof}
	Implication $\ref{it3}\Rightarrow\ref{it2}$ is trivial. Let us prove implication $\ref{it2}\Rightarrow\ref{it1}$. The dual flattening of $\mu^{\smash{(d)}}_{A,\epsilon}$ is given by
	\begin{align} \label{eq:dualFlattening}
		\psi\colon S^{d-1}A &\to A^*,\qquad a_1 \cdots a_{d-1} \mapsto (a_1\buls a_{d-1}) \cdot \epsilon.
	\end{align}
	Clearly the image of this map is contained in $A \cdot \epsilon$. But the conciseness assumption means it is surjective. So $A \cdot \epsilon = \omega_A$ as desired.
	Now we show $\ref{it1}\Rightarrow\ref{it3}$: we know that $\epsilon$ is a dual generator and need to show that \eqref{eq:dualFlattening} is surjective. Take any element of $A^*$, by assumption it can be written as $b \cdot \epsilon$ for some $b \in A$. But then 
	\[
	\psi( 1_A^{ d-2} b)=(1_A\buls 1_{A}\bul b) \cdot \epsilon=b \cdot \epsilon,
	\]
	so $\psi$ is indeed surjective.
	Now let us assume that conditions \ref{it1}--\ref{it3} hold. Then the linear map 
	\[
	\theta_\epsilon\colon A \to A^*,\qquad a \mapsto a \cdot \epsilon,
	\]
	is an isomorphism. Consider the commutative diagram
	\[
	\begin{tikzcd}
		A^{d-1} & A \\ A^{d-1} & A^*
		\arrow["\mu_A^{(d)}", from=1-1, to=1-2]
		\arrow["\id_A^{\times(d-1)}", from=1-1, to=2-1]
		\arrow["\theta_\epsilon", from=1-2, to=2-2]
		\arrow[from=2-1,to=2-2]
	\end{tikzcd}
	\]
	Since the vertical arrows are isomorphisms, the map $\id_{A^*}^{\otimes d-1}\otimes\theta_{\epsilon}$ gives an equivalence between the tensors corresponding to the top and bottom multilinear maps. But these tensors are exactly $\smash\mu_{A}^{\smash{(d-1)}}$ and $\smash\mu_{A,\epsilon}^{\smash{(d)}}$.
\end{proof}
	
This lemma implies in particular that for a given Gorenstein algebra $A$, 
 different choices of a dual generator $\epsilon$ yield equivalent tensors (all equivalent to the partially symmetric tensor $\mu^{\smash{(d-1)}}_A$). So we can think of the evaluation tensor  $\mu^{\smash{(d)}}_{A,\epsilon}$ as \emph{the} symmetric tensor corresponding to the Gorenstein algebra $A$.
 
 \begin{example}\label{exam:mult_x_4}
 	Let $A = \bbC[t] / (t^4)$. This is a Gorenstein algebra of degree $4$, with the socle generated by $t^3$. Let $\{b_1, b_2, b_3, b_4\}$ be the basis of $A$ over $\bbC$, corresponding to $\{1, t, t^2, t^3\}$. 
 	The only nonzero $d$-fold products of these basis vectors are
 	\[
 	b_1^{d} = b_1, \quad b_1^{d-1}b_2 = b_2, \quad b_1^{d-1}b_3=b_1^{d-2}b_2^2=b_3, \quad b_1^{d-1}b_4=b_1^{d-2}b_2b_3=b_1^{d-3}b_2^3=b_4.
 	\]
 	Using \eqref{eq:mult_tensor_explicit}, the iterated multiplication tensor can be written as
 	\begin{align*}
 		\mu_{A}^{\smash{(d)}} &=x_1^{d} \otimes b_1 + dx_1^{d-1}x_2 \otimes b_2 +  dx_1^{d-1}x_3 \otimes b_3 +  \binom{d}{2}x_1^{d-2}x_2^2 \otimes b_3 + dx_1^{d-1}x_4 \otimes b_4\\
 		&\hphantom{{}={}} 
 		+ 2\binom{d-2}{2} x_1^{d-2}x_2x_3 \otimes b_4 + \binom{d}{3}x_1^{d-3}x_2^3 \otimes b_4.
 	\end{align*}
 	As dual generator $\epsilon$ we choose the basis vector $x_4$ dual to $b_4$. We find
 	\begin{equation} \label{eq:example_polynomial}
 		\mu_{A,\epsilon}^{(d)} = dx_1^{d-1}x_4
 		+ 2\binom{d-2}{2} x_1^{d-2}x_2x_3 + \binom{d}{3}x_1^{d-3}x_2^3.
 	\end{equation}
 	Note that %
 	\begin{equation*}%
 		\theta_\epsilon(b_1) = x_4, \quad \theta_\epsilon(b_2)= x_3, \quad \theta_\epsilon(b_3)=x_2, \quad \theta_\epsilon(b_4)=x_1. 
 	\end{equation*} 
 	If we apply the map ${\id_{A^*}^{\otimes d-1}} \otimes \theta_\epsilon$ to the iterated multiplication tensor $\mu_{A}^{\smash{(d-1)}}$, we recover the same polynomial as in \eqref{eq:example_polynomial}. This explicit computation is left to the reader.
 	After rescaling by a factor $d!$ and passing to the divided power basis, the polynomial \eqref{eq:example_polynomial} becomes
 	\begin{equation} \label{eq:example_polynomial_div_power}
 		x_1^{[d-1]}x_4
 		+ x_1^{[d-2]}x_2x_3 + x_1^{[d-3]}x_2^{[3]}.
 	\end{equation}
 \end{example}
 
\subsection{The cactus apolarity lemma} The following result is a symmetric version of the cactus apolarity lemma \cite[Theorem 1.3]{JJ26} (see also \cite[Lemma 5.5]{DM26}). It states that symmetric tensors of minimal \emph{cactus} rank are precisely the multiplication tensors of Gorenstein algebras.

\begin{theorem}\label{thm:multTensorsAndLinearSpanOfSpec}
	Let $f\in S^dV$ be a concise tensor and let $A$ be a $\bbC$-algebra such that $\dim V = \dim A := n <\infty$. Then the following are equivalent:
	\begin{enumerate}
		\item There is an embedding $\iota\colon\Spec A \hookrightarrow \bbP V$ such that $f \in \langle \nu_d(\im \iota) \rangle$;
		\item \label{it:gormult} $A$ is Gorenstein, and $f$ is equivalent to $\mu^{(d-1)}_A$.
	\end{enumerate}
\end{theorem}

\begin{proof}
	Suppose there is an embedding 
	$\iota\colon\Spec A \hookrightarrow \bbP V$ such that 
	$f \in \langle \nu_d (Z) \rangle $, where $Z = \im \iota$ is the projective subscheme corresponding to $\Spec A$. 
	We will prove \eqref{it:gormult}. 
	We view $f \in \langle \nu_d (Z) \rangle$ as a linear map $\bbC \rightarrow \langle \nu_d (Z) \rangle$ that sends $1 \mapsto f$. 
	Then we can write 
	\begin{equation*}
		\bbC \xrightarrow{f} \langle \nu_d (Z) \rangle \hookrightarrow S^d V.
	\end{equation*}
	Taking duals, we get
	\begin{equation*}
		S^d V^* \twoheadrightarrow S^d V^* / I(Z)_d \xrightarrow{f^*} \bbC, 
	\end{equation*}
	where $I(Z)$ is the homogeneous ideal of $Z$. 
	Since $\Spec A$ is a zero-dimensional scheme, its image $Z$ is supported on finitely many points of $\bbP V$. 
	In particular, there is $\ell \in V^*$ such that $Z$ lies in the affine open $D_+(\ell)$ defined by $\ell \neq 0$:
	\begin{equation*}
		D_+(\ell) = \Spec S_{(\ell)} \subseteq \bbP V,
	\end{equation*}
	where $S$ denotes the graded algebra $S^\bullet V^*$ (so $\bbP V = \Proj S$) and $S_{(\ell)}$ denotes the subring of degree $0$ elements in the localized ring $S_{\ell}$. %
	Note that the maps 
	\begin{equation*}
	\eta_j\colon	S^j V^* \to S_{(\ell)}, \quad g \mapsto g / \ell^j
	\end{equation*}
	assemble into an algebra morphism $\eta: S \to S_{(\ell)}$.
	Composing $\eta$ with the algebra morphism $S_{(\ell)} \twoheadrightarrow A$ induced by the embedding $\Spec A \hookrightarrow D_+(\ell)$ gives an algebra morphism $\varphi \colon S \rightarrow A$.
	After restricting to $S^j V^{*}$, we get maps $\varphi_j \colon S^j V^{*} \rightarrow A$. By construction, the kernel of $\varphi_j$ coincides with $I(Z)_j$, so (taking $j=d$) we can write the following diagram: %
	\[\begin{tikzcd}
{S^d V^{*}} & {S^d V^{*}/ I(Z)_d} & \bbC \\
			{S^dA} & A
			\arrow[two heads, from=1-1, to=1-2]
			\arrow["{f^*}", from=1-2, to=1-3]
			\arrow["{\varphi_d}", hook, from=1-2, to=2-2]
			\arrow["{\varphi_1^{\otimes d}}", from=1-1, to=2-1]
			\arrow["\mu_A^{(d)}", two heads, from=2-1, to=2-2]
			\arrow[dashed, from=2-2, to=1-3]
		\end{tikzcd}\]
	\noindent
	Now the conciseness of $f \in \langle \nu_d (Z) \rangle$ implies in particular that $Z$ cannot be contained in a strict linear subspace of $\bbP V$, i.e.\ that $I(Z)_1=0$. Hence $\varphi_1 \colon V^{*} \rightarrow A$ is injective; since by assumption $\dim A = \dim V$ we find that $\varphi_1$ is a bijection. 
	From the above diagram it follows that also 
	$\varphi_d$ is bijective. Let $\varepsilon \colon A \rightarrow \bbC$ be the composition 
	\[
	f^* \circ \varphi_d^{-1} \colon A \rightarrow S^d V^{*}/ I(Z)_d \rightarrow \bbC^*.
	\] 
	Then again by our diagram $\varphi_1^{\otimes d}$ gives an equivalence between the tensors $\mu^{\smash{(d)}}_{A,\varepsilon}$ and $f$. Since $f$ is concise, we conclude by \Cref{lem:GorensteinTensor} that $A$ is Gorenstein. 
	
	The argument can be reversed: if \eqref{it:gormult} holds, there is an isomorphism $\psi: V^* \to A$ such that the following diagram commutes:
	\[\begin{tikzcd}
		{S^d V^{*}} & \bbC \\
		S^d A & 
		\arrow["{f^*}", from=1-1, to=1-2]
		\arrow["\psi^{\otimes d}"', from=1-1, to=2-1]
		\arrow["\mu_{A,\varepsilon}^{(d)}"', from=2-1, to=1-2]
	\end{tikzcd}\]
	The compositions $\mu_{A}^{(j)} \circ \psi^{\otimes j}$ assemble into an algebra morphism $\varphi\colon S\to A$. If we define $\ell := \psi^{-1}(1_A)$, we obtain an algebra morphism $S_{(\ell)} \to A$ mapping $g/\ell^j$ to $\varphi(g)$. This gives our desired embedding \[\Spec A \xhookrightarrow{\iota} D_+(\ell) \subset \bbP V.\] We need to verify that $f \in \langle \nu_d(Z)\rangle$, i.e.\ that $I(Z)_d \subset \ker f^*$, where $Z=\im \iota$. As before we have $I(Z)_d = \ker \varphi_d$, so this readily follows from the diagram above and the definition $\varphi_d = \mu_{A}^{\smash{(d)}} \circ \psi^{\otimes d}$.
\end{proof}

It follows immediately from the theorem that a concise tensor $f$ in $S^dV$ has minimal \emph{smoothable} rank if and only if $f$ is equivalent to $\mu^{(d-1)}_{A}$ for a \emph{smoothable} Gorenstein algebra $A$. 
If $d \geq n-1$, we can invoke \Cref{prop:BB14} and obtain: 

\begin{corollary}\label{cor:main}
	Let $n$ and $d$ be positive integers such that $d \geq n-1$, and let $V$ be an $n$-dimensional $\bbC$-vector space. Then a concise tensor in $f \in S^dV$ has symmetric border rank $n$ if and only if $f \sim \mu^{\smash{(d-1)}}_A$, where $A$ is a smoothable $n$-dimensional Gorenstein algebra. \qedhere
\end{corollary}

\subsection{Recovering the algebra as the centroid} To make the correspondence between minimal border (or smoothable) rank polynomials and smoothable Gorenstein algebras complete, we would like that nonisomorphic algebras give rise to nonequivalent tensors. Note that this fails for $d=2$: %
the multiplication tensor $\mu^{\smash{(1)}}_A$ of any algebra is simply the identity matrix.
However for $d \geq 3$ we can use centroids to recover an algebra from its multiplication tensor.
\begin{proposition}\label{proposition:centroid_isomorphic_A}
	Let $A$ be an algebra and, for any $a\in A$, let $\m_a\colon A\to A$ be the multiplication map defined by $\m_a(b)=a \bul b$ for every $b\in A$. Then, for any $d\geq 3$, the centroid of the multiplication tensor $\mu_{A}^{\smash{(d)}}\in S^{d-1}A^*\otimes A$ is
	\[
	\cen{\mu_{A}^{\smash{(d-1)}}}=\Set{(\transpose{{\m_a}},\m_a)|a\in A}\subseteq \End(A^*)\times \End(A).
	\]
	In particular, $\cen{\mu_{A}^{\smash{(d-1)}}}\cong A$ as algebras.
\end{proposition}
\begin{proof}
	Let $a\in A$. Then, by formula \eqref{definition:comp_i}, we have for $i=1,\ldots,d-1$:
	\begin{align*}
		(\transpose{{\m_a}}\circ_i\mu_{A}^{\smash{(d-1)}})(a_1,\dots,a_{d-1})&=a_1\buls a_{i-1}\bul(a\bul a_i)\bul a_{i+1}\buls a_{d-1}=a\bul(a_1\buls a_{d-1})\\
		&=\m_a (a_1\buls a_{d-1})=({{\m_a}}\circ_{d}\mu_{A}^{\smash{(d-1)}})(a_1,\dots, a_{d-1}),
	\end{align*}
	which implies that $(\transpose{{\m_a}},\m_a)\in 	\cen{\mu_{A}^{\smash{(d-1)}}}$. Conversely, let $(X,Y)\in\cen{\mu_{A}^{\smash{(d-1)}}}$. Then, again by formula \eqref{definition:comp_i}, we have
	\[
	Y(a_1\buls a_{d-1}) = a_1\buls a_{i-1}\bul \transpose{X}(a_i)\bul a_{i+1}\buls a_{d-1}
	\]
	for $i=1,\ldots,d-1$ and any $a_1,\ldots,a_{d-1} \in A$. %
	If we choose $i=1$ and set $a_1=b$ arbitrary and $a_j=1_A$ for every $j=2,\dots,d-1$, we find that $Y(b)=\transpose{X}(b)$ for any $b \in A$. In particular $Y(1_A)=\transpose{X}(1_A) \eqqcolon a$. On the other hand, choosing $i=2$, $a_1=b$ arbitrary, and $a_j=1_A$ for every $j=2,\dots,d-1$ tells us that \[Y(b)=b\bul \transpose{X}(1_A) = b \bul a\] for every $b \in A$. So $Y=\m_a$ and $X=\transpose{Y} = \transpose{{\m_a}}$, completing the proof.
\end{proof}

\section{Classification of minimal border rank polynomials} \label{sec:classification}
 By \Cref{cor:main} and \Cref{proposition:centroid_isomorphic_A}, classifying minimal border rank polynomials in $S^d\bbC^n$ is equivalent to classifying $n$-dimensional smoothable Gorenstein algebras, as long as $d \geq \max\{3,n-1\}$. By \cite{CJN15} all Gorenstein algebras of dimension at most 13 are smoothable. Finite-dimensional algebras up to dimension $7$ have been classified in \cite{Casnati}.
So we can use this classification to list all minimal border rank polynomials in $S^d \bbC^n$ for $n\leq 7$ and $d \geq 3$, except for the cases $(n,d)=(5,3),(6,3),(6,4),(7,3),(7,4),(7,5)$. For the case $(n,d)=(5,3)$ the classification is already known, see \Cref{rmk:S3C5}; we leave the other cases for future work.

\subsection{Reduction to local algebras} It is not hard to see that the iterated multiplication tensor of a product of algebras is the direct sum of the iterated multiplication tensors of the factors. We make this precise in the following lemma (which is phrased in the Gorenstein setting merely for sake of convenience): %

\begin{lemma}
	Let $A$ be a Gorenstein algebra with dual generator $\epsilon$, where $A \cong A_1 \times \cdots \times A_m$. Then every $A_i$ is Gorenstein with dual generator the restriction of $\epsilon$ to $A_i$, and
	\[
	\mu_{A,\epsilon}^{(d)} = \mu_{A_1,\epsilon}^{(d)} + \cdots + \mu_{A_m,\epsilon}^{(d)},
	\]
	where $\smash{\mu_{A_i,\epsilon}^{(d)}}$ is viewed as an element of $\smash{S^d{A^*}}$ via the obvious inclusion $\smash{S^d{A_i^*}} \hookrightarrow \smash{S^d{A^*}}$. 
\end{lemma}
\begin{proof}
	Since $\epsilon$ is a dual generator, it follows that $\smash{a \mapsto \restr{(a \cdot \epsilon)}{A_i}}$ is a surjection $A \twoheadrightarrow \omega_{A_i}$. However, for $j \neq i$, $a \in A_j$ and $b \in A_i$, we have $(a \cdot \epsilon)(b)= \epsilon(a\bul b) = 0$. So our surjection vanishes on each $A_j$ with $j \neq i$, hence its restriction $A_i \to \omega_{A_i}$ must be a surjection. This shows that $\epsilon$ restricts to a dual generator of $A_i$.  
	From the definition of a direct product of algebras, it follows immediately that the restriction of the multilinear map $\smash{\mu_{A}^{(d)}}$ to $A_{i_1} \times \cdots \times A_{i_d}$ is zero, unless $i_1=\cdots=i_d=i$, in which case it is $\smash{\mu_{A_i}^{(d)}}$. The expression for $\smash{\mu_{A,\epsilon}^{(d)}}$ now easily follows.
\end{proof}

In the language of polynomials this can be stated as follows: if $A \cong A_1 \times \cdots \times A_m$ and the $d$-th multiplication tensor of each $A_i$ is equivalent to a polynomial $F_i(x_1,\ldots,x_{n_i})$, then the $d$-th multiplication tensor of $A$ is equivalent to
\begin{equation} \label{eq:polynomial_of_product_algebra}
F_1(x_1, \ldots, x_{n_1}) + F_2(x_{n_1+1},\ldots, x_{n_1+n_2}) + \cdots + F_n(x_{n_1+\cdots+n_{m-1}+1}, \ldots, x_{n_1+\cdots+n_m}).
\end{equation}
A finite-dimensional Gorenstein algebra can be written uniquely as a direct product of local Gorenstein algebras. If $d \geq n-1$, it follows that a polynomial is of minimal border rank if and only if it can be written as a direct sum (as in \eqref{eq:polynomial_of_product_algebra}) of polynomials arising from local Gorenstein algebras, and moreover that if such a decomposition exists it is unique. In the rest of this section we therefore restrict ourselves to classifying minimal border rank polynomials arising from local Gorenstein algebras, we call these \emph{indecomposable}. %

\subsection{Computing iterated multiplication tensors}
In this section we present an algorithm that computes the multiplication tensor of a given local Gorenstein algebra. An implementation in Macaulay2 is presented in \autoref{Appendix}.

\begin{algorithm}\label{algo:evaluationTensor}
The input is an ideal $I \subset S \coloneqq \bbC[t_1,\ldots,t_r]$ whose radical is $(t_1,\ldots,t_r)$; this ensures that $A \coloneqq S/I$ is a finite local algebra. The steps of the algorithm are as follows: 
\begin{enumerate}
	\item Compute %
	monomials $ \{b_1,\ldots,b_n\} \in S$ whose classes modulo $I$ form a basis $B$ 
	of $A=S/I$, with $b_1=1$. %
	This can be done by computing a Gr\"obner basis of $I$, the $b_i$ are the standard monomials of the initial ideal of $I$. 
	\item \label{it:algoStep2} Compute the socle of $A$. If $\Soc(A)$ is not principal, stop here ($A$ is not local Gorenstein). Otherwise, pick a generator, expand it in the basis $B$, and choose one of the basis vectors occurring with a nonzero coefficient. Call this basis vector $b$ and its dual basis vector $\epsilon \in A^*$. %
	\item \label{it:socDeg} Compute the largest $m$ for which $\mfm_A^m \neq 0$. This is the \emph{socle degree} of $A$.
	\item Output the following polynomial:
	\begin{equation}\label{eq:algoOutput}
	\sum_{s=0}^{m}\sum_{\substack{0\leq j_2,\dots,j_n\leq s\\j_2+\cdots+j_n=s}}\epsilon(b_{2}^{\bul j_2}\bul\cdots\bul b_{n}^{\bul j_n})x_1^{[d-s]}x_{2}^{[j_2]}\cdots x_{n}^{[j_n]}.
	\end{equation}
	Here $\epsilon(a)$ is computed in practice by expanding $a \in S/I$ in the basis $B$, and taking the coefficient of $b$.
\end{enumerate} 
\end{algorithm}
\begin{example}
	We return to \Cref{exam:mult_x_4}. The socle is generated by $t^3$ so we must pick $b=t^3=b_4$. The socle degree $m$ is equal to $3$. There are only $3$ nonzero summands appearing in \eqref{eq:algoOutput}, corresponding to the products $b_4=b_2b_3=b_2^3$, and the algorithm correctly outputs \eqref{eq:example_polynomial_div_power}. 	
\end{example}
Note that in the above algorithm we can specify the number $d$, but we can also leave it indeterminate. %
That is, for a fixed $A$ we are really computing all iterated multiplication tensors at the same time. The following proposition shows that our algorithm is correct:
\begin{proposition}
	If $I$ in \Cref{algo:evaluationTensor} defines a Gorenstein algebra $A=S/I$, then the $\epsilon$ computed in {\rm{(\ref{it:algoStep2})}} is a dual generator, and \eqref{eq:algoOutput} is, up to a factor $d!$, equal to the evaluation tensor $\mu^{\smash{(d)}}_{ A,\epsilon}$.
\end{proposition}
\begin{proof}
	By construction, every $b_i$ with $i>1$ lies in the maximal ideal of $A$, hence every product $b_2^{\bul j_2}\bul\cdots\bul b_{n}^{\bul j_n}$ with $j_2+\cdots+j_n>m$ is zero. The result now follows by comparing \eqref{eq:algoOutput} and \eqref{eq:mult_tensor_as_polynomial_div_power}. 
\end{proof}
We ran the algorithm above on all ideals from \cite{Casnati}. The results are summarized in \Cref{tab:MainTable}. %
The ideals in the left column live in a polynomial ring generated by the variables that appear in that ideal.

\begin{table}
	\NiceMatrixOptions{cell-space-top-limit=6pt}
	\resizebox{1\textwidth}{!}{%
\begin{NiceTabular}{p{0.94cm}p{5.2cm}p{10.8cm}}[corners=NW,hvlines]
	\CodeBefore 
	\rowcolors[gray]{2}{0.97}{}[cols=2-3,restart] \Body 
	 & \textbf{Ideal of Gorenstein algebra} & \textbf{Form of minimal border rank}\\ 
	$n=1$ & $(t_1)$ & $x_1^{[d]}$ \\ 
	$n=2$ & $(t_1^2)$ & $x_1^{[d-1]}x_2$ \\ 
	$n=3$ &$(t_1^3)$ & $x_1^{[d-1]}x_3+x_1^{[d-2]}x_2^{[2]}$ \\ 
	\Block{2-1}{$n=4$}
	&$(t_1^2-t_2^2,t_1t_2)$ & $x_1^{[d-1]}x_4+x_1^{[d-2]}x_2^{[2]}+x_1^{[d-2]}x_3^{[2]}$ \\ 
		& $(t_1^4)$ &$x_1^{[d-1]}x_4+x_1^{[d-2]}x_2x_3+x_1^{[d-3]}x_2^{[3]}$
 \\ 
	\Block{3-1}{$n=5$}
	&$(t_1^2-t_i^2)_{2\leq i\leq 3}+(t_it_j)_{1\leq i<j\leq 3}$ & $x_1^{[d-1]}x_5+x_1^{[d-2]}x_2^{[2]}+x_1^{[d-2]}x_3^{[2]}+x_1^{[d-2]}x_4^{[2]}$\\
	&$(t_1^3-t_2^2,t_1t_2)$ & $x_1^{[d-1]}x_5+x_1^{[d-2]}x_4^{[2]}+x_1^{[d-2]}x_2x_3+x_1^{[d-3]}x_2^{[3]}$ \\
	&	$(t_1^5)$ & $x_1^{[d-1]}x_5+x_1^{[d-2]}x_2x_4+x_1^{[d-2]}x_3^{[2]}+x_1^{[d-3]}x_2^{[2]}x_3+x_1^{[d-4]}x_2^{[4]}
	$\\
	\Block{6-1}{$n=6$}
	&$(t_1^2-t_i^2)_{2\leq i\leq 4}+(t_it_j)_{1\leq i<j\leq 4}$ & $x_1^{[d-1]}x_6+x_1^{[d-2]}x_2^{[2]}+x_1^{[d-2]}x_3^{[2]}+x_1^{[d-2]}x_4^{[2]}+x_1^{[d-2]}x_5^{[2]}$\\
	&$(t_1^3-t_1^2t_2, t_2^2)$ & $x_1^{[d-1]}x_4
	+x_1^{[d-2]}x_2x_3+x_1^{[d-2]}x_2x_5+x_1^{[d-2]}x_3x_6+x_1^{[d-3]}x_2^{[2]}x_6\vspace*{0.1cm}\newline+x_1^{[d-3]}x_2^{[3]}
	$\\
	&$(t_1^3-t_i^2)_{2\leq i\leq 3}+(t_it_j)_{1\leq i<j \leq 3}$ & $x_1^{[d-1]}x_6+x_1^{[d-2]}x_2x_3+x_1^{[d-2]}x_4^{[2]}+x_1^{[d-2]}x_5^{[2]}+x_1^{[d-3]}x_2^{[3]}$\\
    &$(t_1^3-t_2^3,t_1t_2)$&$x_1^{[d-1]}x_6+x_1^{[d-2]}x_2x_3
    +x_1^{[d-2]}x_4x_5+x_1^{[d-3]}x_2^{[3]}+x_1^{[d-3]}x_4^{[3]}$\\
	&$(t_1^4-t_2^2,t_1t_2)$ & $x_1^{[d-1]}x_6+x_1^{[d-2]}x_2x_4+x_1^{[d-2]}x_3^{[2]}+x_1^{[d-2]}x_5^{[2]}+x_1^{[d-3]}x_2^{[2]}x_3\vspace*{0.1cm}\newline+x_1^{[d-4]}x_2^{[4]}$\\
	& $(t_1^6)$ & 
	$x_1^{[d-1]}x_6+x_1^{[d-2]}x_2x_5+x_1^{[d-2]}x_3x_4+x_1^{[d-3]}x_2^{[2]}x_4+x_1^{[d-3]}x_2x_3^{[2]}\vspace*{0.1cm}\newline+x_1^{[d-4]}x_2^{[3]}x_3+x_1^{[d-5]}x_2^{[5]}
	$\\
	\Block{9-1}{$n=7$}
	&$(t_1^2-t_i^2)_{2\leq i\leq 5}+(t_it_j)_{1\leq i<j\leq 5}$&
	$x_1^{[d-2]}x_2^{[2]}+x_1^{[d-2]}x_3^{[2]}+x_1^{[d-2]}x_4^{[2]}
	+x_1^{[d-2]}x_5^{[2]}+x_1^{[d-2]}x_6^{[2]}+x_1^{[d-1]}x_7$\\
	&$(t_1^2-t_i^3)_{2\leq i\leq 4}+(t_it_j)_{1\leq i<j\leq 4}$&
	$x_1^{[d-1]}x_7+x_1^{[d-2]}x_2x_3+x_1^{[d-2]}x_4^{[2]}+x_1^{[d-2]}x_5^{[2]}+x_1^{[d-2]}x_6^{[2]}+x_1^{[d-3]}x_2^{[3]}$\\
		&$(t_1^3-t_2^2,t_1^2t_2)$&$x_1^{[d-1]}x_5
		+x_1^{[d-2]}x_2x_7
		+x_1^{[d-2]}x_3^{[2]}
		+x_1^{[d-2]}x_4x_6
		+x_1^{[d-3]}x_2x_6^{[2]}
		\vspace*{0.1cm}\newline+x_1^{[d-3]}x_2^{[2]}x_3
		+x_1^{[d-4]}x_2^{[4]}$\\
	&$(t_1^3-t_2^3, t_1^3-t_3^2)+(t_it_j)_{1\leq i<j\leq 3}$& $x_1^{[d-1]}x_7+x_1^{[d-2]}x_2x_3+x_1^{[d-2]}x_4x_5+x_1^{[d-2]}x_6^{[2]}+x_1^{[d-3]}x_2^{[3]}+x_1^{[d-3]}x_4^{[3]}$\\
	&$( t_1^3-t_1^2t_2, t_1^3-t_3^2, t_1t_3, t_2^2,t_2t_3)$&$x_1^{[d-1]}x_7
	+x_1^{[d-2]}x_2x_3
	+x_1^{[d-2]}x_2x_4
	+x_1^{[d-2]}x_3x_5
	+x_1^{[d-2]}x_6^{[2]}
	\vspace*{0.1cm}\newline+x_1^{[d-3]}x_2^{[2]}x_5
	+x_1^{[d-3]}x_2^{[3]}$\\
		&$(t_1^4-t_i^2)_{2\leq i\leq 3}+(t_it_j)_{1\leq i<j\leq 3}$&
	$x_1^{[d-1]}x_7+x_1^{[d-2]}x_2x_4+x_1^{[d-2]}x_3^{[2]}+x_1^{[d-2]}x_5^{[2]}+x_1^{[d-2]}x_6^{[2]}\vspace*{0.1cm}\newline+x_1^{[d-3]}x_2^{[2]}x_3+x_1^{[d-4]}x_2^{[4]}$
	\\
	&$(t_1^4-t_2^3,t_1t_2)$&$x_1^{[d-1]}x_7
	+x_1^{[d-2]}x_2x_4+x_1^{[d-2]}x_3^{[2]}+x_1^{[d-2]}x_5x_6+x_1^{[d-3]}x_2^{[2]}x_3\vspace*{0.1cm}\newline+x_1^{[d-3]}x_5^{[3]}+x_1^{[d-4]}x_2^{[4]}$\\
	&$(t_1^5-t_2^2,t_1t_2)$& $x_1^{[d-1]}x_7+x_1^{[d-2]}x_2x_5+x_1^{[d-2]}x_3x_4+x_1^{[d-2]}x_6^{[2]}\vspace*{0.1cm}\newline+x_1^{[d-3]}x_2^{[2]}x_4+x_1^{[d-3]}x_2x_3^{[2]}+x_1^{[d-4]}x_2^{[3]}x_3+x_1^{[d-5]}x_2^{[5]}$\\
	&$(t_1^7)$&$x_1^{[d-1]}x_7+x_1^{[d-2]}x_2x_6+x_1^{[d-2]}x_3x_5+x_1^{[d-2]}x_4^{[2]}+x_1^{[d-3]}x_2^{[2]}x_5\vspace*{0.1cm}\newline+x_1^{[d-3]}x_2x_3x_4+x_1^{[d-3]}x_3^{[3]}+x_1^{[d-4]}x_2^{[3]}x_4+x_1^{[d-4]}x_2^{[2]}x_3^{[2]}\vspace*{0.1cm}\newline+x_1^{[d-5]}x_2^{[4]}x_3+x_1^{[d-6]}x_2^{[6]}$
	\end{NiceTabular}
}
\caption{All indecomposable polynomials of minimal border rank in $n \leq 7$ variables, as long as $d \geq \max\{3,n-1\}$.} \label{tab:MainTable}
\end{table}

\begin{remark}\label{rmk:S3C5}
	If $d < n-1$ the polynomials in \Cref{tab:MainTable} still have minimal border rank, after setting any terms with a negative exponent in $x$ to $0$. However, there are additional polynomials of minimal border rank not in the table. For instance, in the case $(n,d)=(5,3)$, the polynomial
	\begin{equation}\label{formula:case_n=5_d=3}
		f=x_1^{[2]}x_3+x_1^{[2]}x_4+x_1x_2x_4+x_2^{[2]}x_4+x_2^{[2]}x_5,
	\end{equation}
	has border rank $5$ but smoothable rank $6$ \cite[Section 4.1]{BB15}. In fact, one can show that up to equivalence this is the only polynomial of minimal border rank $5$ that does not appear in \Cref{tab:MainTable}, see \cite[Theorem 10.5]{LT10} or \cite[Example 4.6]{JLP24}. (The latter reference classifies all order $3$ tensors of minimal border rank $5$ and notes there is a single symmetric case which is not a multiplication tensor.)
\end{remark}

\subsection{An infinite family of minimal border rank polynomials} 
	For $n=8$ there are infinitely many nonisomorphic Gorenstein algebras. This implies that for any degree $d\geq 3$ there are infinitely many nonequivalent polynomials of minimal border rank in $8$ variables. We can construct an explicit family of such algebras as apolar algebras of plane cubics. Let us recall the definition:
	\begin{definition}
		Given a polynomial $f \in \bbC[x_1,\ldots,x_n]$, its \emph{apolar ideal} is the ideal $f^{\perp} \subset \bbC[\partial_1,\ldots,\partial_n]$ of differential operators annihilating $f$. The quotient $\operatorname{ap}(f)=\bbC[\partial_1,\ldots,\partial_n]/{f^{\perp}}$ is called the \emph{apolar algebra}. 
	\end{definition}
	We need the following facts about apolar algebras, see for instance Theorem 3.26 and Remark 3.33 in \cite{Jel22}:
	\begin{itemize}
		\item The apolar algebra $\operatorname{ap}(f)$ of a polynomial is a finite local Gorenstein algebra. 
		\item If $f,g \in \bbC[x_1,\ldots,x_n]_d$ are \emph{homogeneous} polynomials, the apolar algebras $\operatorname{ap}(f)$ and $\operatorname{ap}(g)$ are isomorphic is and only if $f$ and $g$ agree up to a linear change of variables. 
	\end{itemize}
	\begin{example}\label{ex:elliptic}
	For every $a,b \in \bbC\setminus \{0\}$, we can consider homogeneous cubic
	\[
	F_{a,b}(x,y,z) := y^2z-x^3-axz^2-bz^3 \in \bbC[x,y,z]_3.
	\]
	Two such cubics $F_{a,b}$ and $F_{a',b'}$ agree up to a linear change of variables if and only if $b^2/a^3 = b'^2/a'^3$ (this is well known from the theory of elliptic curves, but also not hard to verify by hand). The apolar ideal of $F_{a,b}$ is given by
	\[
	F_{a,b}^\perp = (\partial_x\partial_y,a\partial_y^2+\partial_x\partial_z,a\partial_x^2-3\partial_z^2-9b\partial_y^2).
	\] 
	Using \Cref{algo:evaluationTensor}, we can compute the corresponding polynomial of minimal border rank as
	\begin{align*}
	d!\mu_{F_{a,b}^{\perp},\epsilon}^{(d)} = \frac{1}{3b}\Bigl( &3x_1^{[d-3]}x_2^{[3]}-x_1^{[d-2]}x_4x_5+ax_1^{[d-2]}x_3x_6-x_1^{[d-3]}x_4^{[2]}x_6\\ &+ax_1^{[d-3]}x_2x_6^{[2]}+3bx_1^{[d-3]}x_6^{[3]}+ax_1^{[d-2]}x_2x_7+3bx_1^{[d-2]}x_6x_7+3bx_1^{[d-1]}x_8\Bigr).
	\end{align*}

	For every fixed $d\geq 3$, this yields infinitely many nonequivalent polynomials of minimal border rank in 8 variables. 
	
\end{example}

\begin{remark} \label{rmk:8_9_variables}
{In \cite{Casnati}, there is not only a list of all Gorenstein algebras of dimension $n\leq 7$, but also explicit classifications of Gorenstein algebras of dimensions $8$ and $9$. These classifications each consist of a finite list of algebras, and one (if $n=8$) or two (if $n=9$) one-parameter families. The one-parameter family provided for $n=8$ agrees up to isomorphism with the one from \Cref{ex:elliptic}. We used \Cref{algo:evaluationTensor} to compute also these cases, providing a complete classification of minimal border rank polynomials in $n\leq 9$ variables (and $d \geq n-1$ as usual). The full list of these polynomials for $n=8,9$ is given in the auxiliary files.}
\end{remark}

\subsection{Testing minimal border rank} 
 The following result by Jagie{\l}{\l}a and Jelisiejew characterizes multipication tensors of Gorenstein algebras in terms of centroids and $1$-genericity:
 \begin{proposition}[{\cite[Corollary 4.14]{JJ26}}]\label{proposition_JJ26}
 	Let $d\geq 3$ and let $T\in V^{\otimes d}$ be a concise tensor such that $\dim(\cen{T})\geq \dim V$. Then the following are equivalent:
 	\begin{enumerate}[label=(\roman*)]
 		\item	$T$ is 1-generic;
 		\item $T\sim \mu_{A,\epsilon}^{\smash{(d)}}$ for some algebra $A$ and some $\epsilon\in \omega_A$.
 	\end{enumerate}
 	If these conditions hold, then $A$ is Gorenstein, $A\cong\cen{T}$, and $\omega_A=A\cdot\epsilon$.
 \end{proposition}
In the symmetric case, as long as $n \leq 13$ and $d \geq \max\{3,n-1\}$, having minimal border rank is the same as being equivalent to $\mu_{A}^{\smash{(d-1)}}$ for some Gorenstein algebra $A$. So in this regime, we can use \Cref{proposition_JJ26} to give a practical test for verifying whether a given polynomial has minimal border rank: 
\begin{proposition}\label{proposition:equivalence_Gorenstein_algebra}
	Let $f\in S^dV$ be a concise tensor. Then the following are equivalent:
	\begin{enumerate}[label=(\roman*)]
		\item\label{prop_equiv_item(1)} $\dim (\cen{f})\geq \dim V$ and $\det\bigl(\Hess(f)\bigr)\neq 0$;
		\item\label{prop_equiv_item(2)} $f\sim \mu_{A}^{\smash{(d-1)}}$ for some Gorenstein algebra $A$.
	\end{enumerate}
\end{proposition}
\begin{proof}%
	Let $\dim (\cen{f})\geq \dim V$ and $\Hess(f)\neq 0$. By \Cref{proposition_HMV20}, we have that $f$ is $1$-generic, so we can apply \Cref{proposition_JJ26} and \Cref{lem:GorensteinTensor} to obtain item \ref{prop_equiv_item(2)}. 
	Conversely, if \ref{prop_equiv_item(2)} holds, then %
	by \Cref{lem:GorensteinTensor} we have $f\sim \mu_{A,\epsilon}^{\smash{(d)}}$ where $\epsilon$ is a dual generator. We have in particular $V\cong A^*$ and hence, by \Cref{proposition:centroid_isomorphic_A},
	\[
	\dim (\cen{f})=\dim (\cen{\mu_{A,\epsilon}^{\smash{(d)}}})=\dim A = \dim V.
	\]
	Moreover, by \Cref{proposition_JJ26}, $f$ is $1$-generic, which by \Cref{proposition_HMV20} is equivalent to $\det\bigl(\Hess(f)\bigr)\neq 0$. %
\end{proof}

For a given polynomial $f$, we can compute its centroid using \Cref{lem:centroidHessian}. So \Cref{proposition:equivalence_Gorenstein_algebra}\ref{prop_equiv_item(1)} can be checked computationally, yielding a practical algorithm to verify whether or not a given polynomial is equivalent to the evaluation tensor of a Gorenstein algebra. We implemented this in Macaulay2, see \Cref{Appendix}.

\begin{example}\label{example:dim4table}
	We applied  \Cref{proposition:equivalence_Gorenstein_algebra}\ref{prop_equiv_item(1)} to the six polynomials in \cite[Theorem 10.4]{LT10}, which classifies minimal border rank polynomials for $n=4$. Two of them in fact do not meet the conditions: $x^{d-2}yz$ has vanishing Hessian (because it is not a concise tensor). For the polynomial
	\[
	f = x^{d-3}y^3+x^{d-2}z^2+x^{d-1}w
	\]
	we compute its Hessian $H$ to be
	\[
	\resizebox{1\textwidth}{!}{$
	\begin{pNiceMatrix}
		 (d-3)(d-4)x^{d-5}y^3 + (d-2)(d-3)x^{d-4}z^2 + (d-1)(d-2)x^{d-3}w & 3(d-3)x^{d-4} y^2 & 2(d-2)x^{d-3}z & (d-1)x^{d-2} \\
		 3(d-3)x^{d-4} y^2 & 6x^{d-3}y & 0 & 0 \\
		 2(d-2)x^{d-3}z & 0 & 2^{d-2} & 0 \\
		 (d-1)x^{d-2} & 0 & 0 & 0
	\end{pNiceMatrix}
	$}
	\]
	The determinant of $H$ is nonzero, but the equation $YH=H\transpose{Y}$ has a 3-dimensional space of solutions, so the  centroid of $f$ has dimension $3<4$. We conclude that $f$ is not an evaluation tensor, and hence does not have minimal border rank. The remaining 4 polynomials do satisfy the requirements, for the computation we refer to our code.
\end{example}

\appendix
\section{Code}\label{Appendix}
The Macaulay2 code accompanying this article is attached as auxiliary files.
\begin{itemize}
	\item \texttt{PolynomialsOfMinimalBorderRank.m2} is a Macaulay2 package, which can be loaded using \texttt{loadPackage "PolynomialsOfMinimalBorderRank"}. It provides an implementation of \Cref{algo:evaluationTensor} to construct the evaluation tensor of a given Gorenstein algebra, as well as an algorithm which checks whether a given polynomial is an evaluation tensor using \Cref{proposition:equivalence_Gorenstein_algebra}.
	\item \texttt{PolynomialsOfMinimalBorderRank\_List} contains the list of minimal border rank forms in up to $7$ variables, as presented in \Cref{tab:MainTable}. They are presented in a format independent of the degree $d$; in order to get the actual polynomials one needs to multiply them by a suitable power of $x_1$ and interpret the result in the divided power basis. 
	\item Similarly, \texttt{PolynomialsOfMinimalBorderRank\_List\_n} (for $n=8,9$) contain the lists of minimal border rank forms in $8$ (respectively $9$) variables. See \Cref{rmk:8_9_variables}.
	\item The three files whose name starts with \texttt{PolynomialsOfMinimalBorderRank\_MakeList} contain the Macaulay2 code used to the generate the aformentioned lists. The code essentially consists of applying our package to the classification of Gorenstein algebras in \cite{Casnati}.
	\item \texttt{PolynomialsOfMinimalBorderRank\_InfiniteFamily.m2} is used to compute the polynomial of \Cref{ex:elliptic}.
	\item \texttt{PolynomialsOfMinimalBorderRank\_Test.m2} illustrates \Cref{proposition:equivalence_Gorenstein_algebra} on the polynomials in \cite[Theorems 10.4 and 10.5]{LT10}, see \Cref{example:dim4table}.
\end{itemize}

\emergencystretch=1em
\printbibliography

@misc{AM69,
	author = {Atiyah, Michael F. and Macdonald, I. G.},
	title = {Introduction to commutative algebra},
	year = {1969},
	howpublished = {Reading, {Mass}.-{Menlo} {Park}, {Calif}.-{London}-{Don} {Mills}, {Ont}.: {Addison}-{Wesley} {Publishing} {Company} (1969)},
	zbMATH = {3279238},
	Zbl = {0175.03601}
}

@article{BB14,
	author = {Buczy{\'n}ska, Weronika and Buczy{\'n}ski, Jaros{\l}aw},
	title = {Secant varieties to high degree {Veronese} reembeddings, catalecticant matrices and smoothable {Gorenstein} schemes},
	fjournal = {Journal of Algebraic Geometry},
	journal = {J. Algebr. Geom.},
	issn = {1056-3911},
	volume = {23},
	number = {1},
	pages = {63--90},
	year = {2014},
}

@article{BB15,
	author = {Buczy{\'n}ska, Weronika and Buczy{\'n}ski, Jaros{\l}aw},
	title = {On differences between the border rank and the smoothable rank of a polynomial},
	fjournal = {Glasgow Mathematical Journal},
	journal = {Glasg. Math. J.},
	volume = {57},
	number = {2},
	pages = {401--413},
	year = {2015},
	}

@article{Ber+18,
	TITLE = {The hitchhiker guide to: secant varieties and tensor decomposition},
	AUTHOR = {Bernardi, A. and Carlini, E. and Catalisano, M.~V. and Gimigliano, A. and Oneto, A.},
	JOURNAL = {Mathematics},
	VOLUME = {6},
	NUMBER = {12},
	YEAR = {2018},
	PAGES = {Paper no.~314, 86 pp.},     
}

@article{BGI11,
	title={Computing symmetric rank for symmetric tensors},
	author={Bernardi, A. and Gimigliano, A. and Idà, M.},
	journal={J. Symbolic Comput.},
	YEAR={2011},
	volume={46},
	number={1},
	pages={34--53},
}

@article{BJ17,
	author = {Buczy{\'n}ski, Jaros{\l}aw and Jelisiejew, Joachim},
	title = {Finite schemes and secant varieties over arbitrary characteristic},
	fjournal = {Differential Geometry and its Applications},
	journal = {Differ. Geom. Appl.},
	issn = {0926-2245},
	volume = {55},
	pages = {13--67},
	year = {2017},
	doi = {10.1016/j.difgeo.2017.08.004},
	zbMATH = {6810581},
	Zbl = {1391.14094}
}

@article {BMW20,
	AUTHOR = {Brooksbank, P. A. and Maglione, J. and Wilson, J. B.},
	TITLE = {Exact sequences of inner automorphisms of tensors},
	JOURNAL = {J. Algebra},
	FJOURNAL = {Journal of Algebra},
	VOLUME = {545},
	YEAR = {2020},
	PAGES = {43--63},
}

@misc{CFJ25,
	author = {Canino, S. and Flavi, C. and Jelisiejew, J.},
	title = {Detecting {Direct} {Sums} of {Tensors} and {Their} {Limits}},
	year = {2025},
	howpublished = {preprint, {arXiv}:2512.05215 [math.{AG}]},
}

@article {CGLM08,
    AUTHOR = {Comon, P. and Golub, G. and Lim, L.-H. and Mourrain,
              B.},
     TITLE = {Symmetric tensors and symmetric tensor rank},
   JOURNAL = {SIAM J. Matrix Anal. Appl.},
  FJOURNAL = {SIAM Journal on Matrix Analysis and Applications},
    VOLUME = {30},
      YEAR = {2008},
    NUMBER = {3},
     PAGES = {1254--1279},
}

@incollection {CGO14,
	author = {Carlini, Enrico and Grieve, Nathan and Oeding, Luke},
	title = {Four lectures on secant varieties},
	booktitle = {Connections between algebra, combinatorics, and geometry. Selected papers based on the presentations at the workshop, Regina, Canada, May 29 -- June 1, 2012, the special session on interactions between algebraic geometry and commutative algebra, Regina, Canada, June 2--3, 2012 and the conference on further connections between algebra and geometry, Fargo, ND, USA, February 2--3, 2013},
	pages = {101--146},
	year = {2014},
	publisher = {New York, NY: Springer},
}

@article{CJN15,
	author = {Casnati, Gianfranco and Jelisiejew, Joachim and Notari, Roberto},
	title = {Irreducibility of the {Gorenstein} loci of {Hilbert} schemes via ray families},
	fjournal = {Algebra \& Number Theory},
	journal = {Algebra Number Theory},
	issn = {1937-0652},
	volume = {9},
	number = {7},
	pages = {1525--1570},
	year = {2015},
	doi = {10.2140/ant.2015.9.1525},
	zbMATH = {6493522},
	Zbl = {1349.14011}
}

@article{Com94,
	author={Comon, P.},
	title={Independent component analysis, a new concept?},
	journal={Signal Process.},
	volume={36},
	YEAR={1994},
	pages={287--314},
}

@misc{DM26,
	title={Nonlinear methods for tensors: determinantal equations for secant varieties beyond cactus}, 
	author={Matěj Doležálek and Mateusz Michałek},
	year={2026},
	howpublished = {preprint, {arXiv}:2602.12762 [math.{AG}]},  
}

@book{Eis95,
	author = {Eisenbud, David},
	title = {Commutative algebra. {With} a view toward algebraic geometry},
	fseries = {Graduate Texts in Mathematics},
	series = {Grad. Texts Math.},
	issn = {0072-5285},
	volume = {150},
	isbn = {3-540-94269-6; 3-540-94268-8},
	year = {1995},
	publisher = {Berlin: Springer-Verlag},
	zbMATH = {704831},
	Zbl = {0819.13001}
}

@article {Fla25a,
	AUTHOR = {Flavi, C.},
	TITLE = {Decompositions of powers of quadrics},
	YEAR = {2025},
	JOURNAL = {Diss. Math.},
	VOLUME = {602},
	PAGES = {108 pp.}
}

@article{Gre+02,
	author = {Grellier, Olivier and Comon, Pierre and Mourrain, Bernard and Tr{\'e}buchet, Philippe},
	title = {Analytical blind channel identification},
	fjournal = {IEEE Transactions on Signal Processing},
	journal = {IEEE Trans. Signal Process.},
	volume = {50},
	number = {9},
	pages = {2196--2207},
	year = {2002},
}

@article {Hou77,
	AUTHOR = {Ho{\"u}el, M.~J.},
	TITLE = {Mélanges du role de l'expérience dans les sciences exactes},
	JOURNAL = {J. Math. Élém.},
	VOLUME = {1},
	YEAR = {1877},
	LANGUAGE = {French},
	PAGES = {118--128},
}

@article {HMV20,
	AUTHOR = {Huang, Hang and Micha\l ek, Mateusz and Ventura, Emanuele},
	TITLE = {Vanishing {H}essian, wild forms and their border {VSP}},
	JOURNAL = {Math. Ann.},
	FJOURNAL = {Mathematische Annalen},
	VOLUME = {378},
	YEAR = {2020},
	NUMBER = {3-4},
	PAGES = {1505--1532},
}

@misc{Jel22,
	author = {Joachim Jelisiejew},
	title = {Hilbert schemes of points and applications},
	year = {2022},
	howpublished = {preprint, {arXiv}:2205.10584 [math.{AG}]},
	arXiv = {arXiv:2205.10584}
}

@incollection{Jel24,
	author = {Jelisiejew, Joachim},
	title = {Open problems in deformations of {Artinian} algebras, {Hilbert} schemes and around},
	booktitle = {Deformation of Artinian algebras and Jordan type. AMS-EMS-SMF special session, Universit\'e Grenoble Alpes, Grenoble, France, July 18--22, 2022},
	pages = {3--25},
	year = {2024},
	publisher = {Providence, RI: American Mathematical Society (AMS)},
	}

@misc{JJ26,
	title={Unrestrictions and concise secant varieties}, 
	author={Jagie{\l}{\l}a, Jakub and Jelisiejew, Joachim},
	year={2026},
	howpublished = {preprint, {arXiv}:2604.24879 [math.{AG}]}, 
}

@article{JLP24,
	author = {Jelisiejew, Joachim and Landsberg, J. M. and Pal, Arpan},
	title = {Concise tensors of minimal border rank},
	fjournal = {Mathematische Annalen},
	journal = {Math. Ann.},
	issn = {0025-5831},
	volume = {388},
	number = {3},
	pages = {2473--2517},
	year = {2024},
	doi = {10.1007/s00208-023-02569-y},
	zbMATH = {7808055}
}

@Book{Landsberg12,
	author    = {Landsberg, J. M.},
	publisher = {Providence, RI: American Mathematical Society (AMS)},
	title     = {Tensors: geometry and applications},
	year      = {2012},
	isbn      = {978-0-8218-6907-9},
	series    = {Grad. Stud. Math.},
	volume    = {128},
	fseries   = {Graduate Studies in Mathematics},
	issn      = {1065-7339},
	zbl       = {1238.15013},
	zbmath    = {5968745},
}

@book{Lan17,
	author = {Landsberg, J. M.},
	title = {Geometry and complexity theory},
	fseries = {Cambridge Studies in Advanced Mathematics},
	series = {Camb. Stud. Adv. Math.},
	volume = {169},
	year = {2017},
	publisher = {Cambridge: Cambridge University Press},
}

@book{Leb59,
	title={Exercices d'analyse numérique},
	author={Lebesgue,V.~A.},
	language={French},
	YEAR={1859},
	publisher={Leiber et Faraguet, Éditeurs},
	address={Paris},
}

@article {LO13,
	AUTHOR = {Landsberg, J. M. and Ottaviani, G.},
	TITLE = {Equations for secant varieties of {V}eronese and other varieties},
	JOURNAL = {Ann. Mat. Pura Appl. (4)},
	FJOURNAL = {Annali di Matematica Pura ed Applicata. Series IV},
	VOLUME = {192},
	YEAR = {2013},
	NUMBER = {4},
	PAGES = {569--606},
}

@article{LT10,
	author   = {Landsberg, J. M. and Teitler, Zach},
	journal  = {Found. Comput. Math.},
	title    = {On the ranks and border ranks of symmetric tensors},
	year     = {2010},
	issn     = {1615-3375},
	number   = {3},
	pages    = {339--366},
	volume   = {10},
	doi      = {10.1007/s10208-009-9055-3},
	fjournal = {Foundations of Computational Mathematics},
	zbl      = {1196.15024},
	zbmath   = {5710612},
}

@article{Luc76,
	title={Sur la théorie des nombres},
	author={Lucas,E.},
	journal={Nouv. Corresp. Math.},
	language={French},
	YEAR={1876},
	volume={2},
	pages={101--105},
}

@misc{MS,
	author = {Grayson, Daniel R. and Stillman, Michael E.},
	title = {Macaulay2, a software system for research in algebraic geometry},
	howpublished = {Available at \url{http://www.macaulay2.com}}
}

@book{Mum95,
    AUTHOR = {Mumford, D.},
     TITLE = {Algebraic geometry, {I}},
    series = {Classics in Mathematics},
  subtitle = {Complex projective varieties},
      NOTE = {reprint of the 1976 edition},
 publisher = {Springer-Verlag},
   address = {Berlin},
      YEAR = {1995},
}

@misc{MV24,
	title={Symmetrization maps and minimal border rank Comon's conjecture}, 
	author={Tomasz Mańdziuk and Emanuele Ventura},
	year={2024},
	howpublished = {preprint, {arXiv}:2411.05721 [math.{AG}]} 
}

@article{Myasnikov,
	author = {Myasnikov, A. G.},
	title = {Definable invariants of bilinear mappings},
	fjournal = {Siberian Mathematical Journal},
	journal = {Sib. Math. J.},
	volume = {31},
	number = {1},
	pages = {89--99},
	year = {1990},
	note = {Translation from Sib. Mat. Zh. 31, No. 1(179), 104-115 (1990)},
}

@article{SGB00,
	author={Sidiropoulos, N.~D. and Giannakis, G.~B. and Bro, R.},
	title={Blind parafac receivers for ds-cdma systems},
	journal={IEEE Trans. Signal Process.},
	volume={48},
	YEAR={2000},
	pages={810--823},
}

@article{Sid+17,
	author = {Sidiropoulos, Nicholas D. and de Lathauwer, Lieven and Fu, Xiao and Huang, Kejun and Papalexakis, Evangelos E. and Faloutsos, Christos},
	title = {Tensor decomposition for signal processing and machine learning},
	fjournal = {IEEE Transactions on Signal Processing},
	journal = {IEEE Trans. Signal Process.},
	volume = {65},
	number = {13},
	pages = {3551--3582},
	year = {2017},
}

@article{Str83,
	author = {Strassen, V.},
	title = {Rank and optimal computation of generic tensors},
	journal = {Linear Algebra Appl.},
	volume = {52},
	YEAR = {1983},
	pages = {645--685},
}

@inproceedings{Val01,
	AUTHOR = {Valiant, L.~G.},
	TITLE = {Quantum computers that can be simulated classically in polynomial time},
	BOOKTITLE = {\textit{Proceedings of the Thirty-Third Annual ACM Symposium on Theory
	of Computing} ({H}ersonissos, 2001), {ACM}, {N}ew {Y}ork, {H}ersonissos},
	YEAR = {2001},
	PAGES = {114--123},
}

@book{Wig19,
	author = {Wigderson, Avi},
	title = {Mathematics and computation. {A} theory revolutionizing technology and science},
	year = {2019},
	publisher = {Princeton, NJ: Princeton University Press},
}

@article {Wil12,
	AUTHOR = {Wilson, J. B.},
	TITLE = {Existence, algorithms, and asymptotics of direct product decompositions, {I}},
	JOURNAL = {Groups Complex. Cryptol.},
	FJOURNAL = {Groups. Complexity. Cryptology},
	VOLUME = {4},
	YEAR = {2012},
	NUMBER = {1},
	PAGES = {33--72},
}

@article{Casnati,
	author = {Casnati, Gianfranco},
	title = {Isomorphism types of {Artinian} {Gorenstein} local algebras of multiplicity at most 9},
	fjournal = {Communications in Algebra},
	journal = {Commun. Algebra},
	issn = {0092-7872},
	volume = {38},
	number = {8},
	pages = {2738--2761},
	year = {2010},
}

@unpublished{JelBedlewo,
	author = {Jelisiejew, Joachim},
	title = {Moduli of tensors},
	year = {2025},
	note = {Notes from the MFO Graduate Seminar 2025. In preparation}
}
\end{document}